%% file: cotilt-derived-eq_v4.tex
\documentclass[a4paper,reqno]{amsart}
\usepackage{amssymb,amscd,amsthm,amsmath}
\usepackage[czech,english]{babel}
\usepackage{dsfont}
\usepackage{url}
\usepackage[all]{xy}
\usepackage{hyperref}

\input{macros.tex}

\title{Derived equivalences induced by big cotilting modules}

\author{Jan \v{S}\v{t}ov\'\i\v{c}ek}
\address{Charles University, Faculty of Mathematics and Physics, Department of Algebra \\
Sokolovsk\'{a} 83, 186 75 Prague 8, Czech Republic}
\email{stovicek@karlin.mff.cuni.cz}


\date{\today}
\thanks{This research was supported by GA~\v{C}R P201/12/G028.}

\begin{document}
\begin{abstract}
We prove that given a Grothendieck category $\mcG$ with a tilting object of finite projective dimension, the induced triangle equivalence sends an injective cogenerator of $\mcG$ to a big cotilting module. Moreover, every big cotilting module can be constructed like that in an essentially unique way. We also prove that the triangle equivalence is at the base of an equivalence of derivators, which in turn is induced by a Quillen equivalence \wrt suitable abelian model structures on the corresponding categories of complexes.
\end{abstract}

\maketitle

\setcounter{tocdepth}{1}
\tableofcontents

\section*{Introduction}

If $\mcG$ is a Grothendieck category with a tilting object $T \in \mcG$ (i.e.\ $T$ is a rigid compact generator for the unbounded derived category of $\mcG$) and $R$ is the endomorphism ring of $T$, then we have a triangle equivalence from $\Der\mcG$ to $\Der R$ which sends $T$ to $R$. A motivating question for this paper, answered in~\cite{HRS96,CGM07} for tilting objects of projective dimension $1$, is when we can go back. That is, if we start with a ring $R$, under what circumstances can we find a Grothendieck category $\mcG$ with a tilting object $T$ and a triangle equivalence $\Der R \to \Der\mcG$ sending $R$ to~$T$?

\smallskip

Our approach is inspired by classical representation theory where, in the context of module categories over artin algebras, a finitely generated injective cogenerator of $\mcG$ is always sent to a finitely generated cotilting module. The general situation is analogous, but an injective cogenerator of an arbitrary Grothendieck category $\mcG$ is far from being finitely generated in any reasonable sense. The solution is to pass to so-called big cotilting modules, introduced in~\cite{ACo01}. Now we can explain our first result (Theorem~\ref{thm:cotilt-from-tilt}): If $\mcG$ is a Grothendieck category, $T$ is a tilting object of finite projective dimension (in the sense that $\Ext^i_\mcG(T,-) \equiv 0$ for $i \gg 0$) and $W \in \mcG$ is an injective cogenerator, then the derived equivalence $\Der\mcG \to \Der R$ sends $W$ to a big cotilting $R$-module $C$.

This is, however, not the end of the story. We also prove that every big cotilting module occurs like that in an essentially unique way (Theorems~\ref{thm:heart-Grothendieck} and~\ref{thm:bijection}). In order to do so, we employ a model structure associated with the cotilting module (Theorem~\ref{thm:cotilt-models}), which is constructed using a technique discovered by Hovey and Gillespie, \cite{Hov02,Gil04}. This model structure is very handy for reconstructing the Grothendieck category $\mcG$ from the cotilting module $C \in \ModR$. Not surprisingly, we construct $\mcG$ as the heart of a certain $t$-structure in $\Der R$ (Theorem~\ref{thm:cotilt-t-structure}).

Along the way, we obtain a lot of extra information. Remarkably, there is a much tighter connection between $\mcG$ and $\ModR$ than only the derived equivalence. It is a rather easy observation that $\Hom_\mcG(T,-)\dd \Cpx\mcG \to \Cpx R$ is a right Quillen equivalence (Theorem~\ref{thm:QE}) if we choose suitable model structures on~$\Cpx\mcG$ and~$\Cpx R$.

There is a not so well-known consequence of the latter fact. Both $\Der\mcG$ and $\Der R$ are at the base of a Grothendieck derivator, and it follows that the corresponding derivators must be equivalent. In pedestrian terms, we have a derived equivalence $\Der{\mcG^I} \to \Der{\ModR^I}$ for every diagram $I$ and all these equivalences are compatible with the scalar restriction functors. In fact, we prove the equivalence of derivators directly (Theorem~\ref{thm:equiv-derivators}) in a way which is practical for our computations. This seems to have been a neglected fact so far, even in the case considered by Happel, Reiten, Smal\o{}~\cite{HRS96} and Colpi, Gregorio and Mantese~\cite{CGM07}. As a consequence, we can for instance prove that the tilting object $T \in \mcG$ must be finitely presentable in~$\mcG$ (Theorem~\ref{thm:comp-gen-to-fp}).

\smallskip

Although some questions are answered, others come to ones mind. For instance, what if we consider derived equivalences between two Grothendieck categories? This is a situation which one encounters in algebraic geometry, see~\cite[Ch. 7]{AHK07}. The injective cogenerator of one of the categories would then be sent to what is best called a big cotilting object or even a big cotilting complex. As opposed to tilting complexes in the sense of Rickard~\cite{Rick89} or big cotilting modules~\cite{ACo01}, neither of these two concepts seems to have been even defined.

\subsection{Acknowledgment}

I would like to thank Ivo Herzog for suggesting an alternative construction of $\mcG$ via the functor category, which considerably simplified the proof of Theorem~\ref{thm:heart-Grothendieck}.

\section{Big cotilting modules from injective cogenerators}
\label{sec:inj-cogen-to-cotilt}

We start with explaining our motivation which shows how big cotilting modules occur naturally. Suppose that $\mcG$ is a Grothendieck category. We will denote by $\Cpx\mcG$ the category of all cochain complexes
\[ \cdots \la X^{-1} \mapr{\dif^{-1}} X^0 \mapr{\dif^0} X^1 \mapr{\dif^1} X^2 \la \cdots \]
over $\mcG$, by $\Htp\mcG$ the corresponding homotopy category of complexes, and by $\Der\mcG$ the unbounded derived category of $\mcG$. If $\mcG = \ModR$ is the category of right $R$-modules over a unital ring $R$, we shorten the notation to $\Cpx R$, $\Htp R$ and $\Der R$. 
We recall a crucial notion for our discussion.

\begin{defn} \label{defn:tilt}
Let $\mcG$ be a Grothendieck category and $T \in \mcG$ be an object. Then $T$ is called \emph{tilting} if there is a ring $R$ and a triangle equivalence
\[ \Der \mcG \la \Der R \]
which sends $T$ to $R$.
\end{defn}

\begin{rem} \label{rem:classical}
Since we will deal with big cotilting modules later on, we will often call objects $T$ as above \emph{classical tilting} objects to stress that they must be small in a suitable sense. This means that $T$ is compact in $\Der\mcG$ (see Proposition~\ref{prop:tilt-char} just below) and later in~\S\ref{subsec:htpy-finite} we will prove that under appropriate assumptions $T$ must be finitely presentable in $\mcG$ in the sense of~\cite{AR94,GU71}.
\end{rem}

The purpose oriented definition above can be translated, using standard results, to the properties of $T$ which are often easier checked. For later reference, we first define the related terminology.

\begin{defn} \label{defn:comp-gen}
Let $\mcT$ be a triangulated category with small coproducts. An object $Z \in \mcT$ is \emph{compact} if, given any small collection $(X_i \mid i \in I)$ of objects of $\mcT$, the canonical homomorphism of abelian groups
\[ \coprod_i \Hom_\mcT(Z,X_i) \la \Hom_\mcT\Big(Z,\coprod_i X_i\Big) \]
is an isomorphism. Equivalently, every morphism $Z \to \coprod_i X_i$ in $\mcT$ has only finitely many non-zero components.

An object $Z \in \mcT$ is a \emph{generator} of $\mcT$ (in the triangulated sense) if, given a non-zero object $0 \ne X \in \mcT$, there exists $i \in \bbZ$ \st  $\Hom_\mcT(\Sigma^i Z,X) \ne 0$.
\end{defn}

\begin{prop} \label{prop:tilt-char}
Let $\mcG$ be a Grothendieck category and $T \in \mcG$ be an object. Then $T$ is tilting \iff the following three properties are satisfied:
\begin{enumerate}
\item[(T1)] $T$ is a compact object in $\Der\mcG$.

\item[(T2)] $T$ is rigid. That is, $\Ext^i_{\mcG}(T,T) = 0$ for each $i > 0$.

\item[(T3)] $T$ is a generator of $\Der\mcG$.
\end{enumerate}

In such a case we have $R \cong \End_\mcG(T)$, and if $R = \End_\mcG(T)$, then the derived functor $\RHom_\mcG(T,-)\dd \Der\mcG \la \Der R$ is a triangle equivalence sending $T$ to $R$.
\end{prop}

\begin{proof}
If $T$ is tilting, then $T$ must satisfy conditions (T1)--(T3) as $R$ is well known to satisfy the analogous conditions in $\Der R$ and clearly $R \cong \End_R(R) \cong \End_\mcG(T)$.

The converse goes back to~\cite{CPS86,Hap87,Kel94,Rick89}; see also~\cite[\S6]{AJS03}. Referring the sources above for details, we can argue as follows. Let us set $R = \End_\mcG(T)$. Then the functor
\[ \RHom_\mcG(T,-)\dd \Der\mcG \la \Der R \]
sends $T$ to a complex whose $i$-th cohomology is isomorphic to $\Ext^i_{\mcG}(T,T)$ for all $i\in\bbZ$, where we use the standard convention that negative Ext groups vanish. Because of (T2), $\RHom_\mcG(T,T)$ is isomorphic to $R$ in $\Der R$. Hence $\RHom_\mcG(T,-)$ sends the compact generator $T$ to the compact generator $R$, preserves coproducts and induces isomorphisms
\[ \Hom_{\Der\mcG}(T,\Sigma^i T) \la \Hom_{\Der R}(R,\Sigma^i R) \]
for all $i \in \bbZ$. This already formally implies that $\RHom_\mcG(T,-)$ is a triangle equivalence.
\end{proof}

\begin{expl} \label{expl:beilinson}
A well known class of examples of Grothendieck categories with a tilting object is due to Be{\u\i}linson~\cite{Bei78}. We put $\mcG = \Qco{\PP{n}{k}}$, where $n \ge 1$ and $k$ is a field. The tilting object is $T = \OO \oplus \OO(1) \oplus \dots \oplus \OO(n)$, a coproduct of the $n+1$ twists of the structure sheaf, and $R = \End_\mcG(T)$ is a finite dimensional algebra over $k$, nowadays called the \emph{Be{\u\i}linson algebra}. However, there exist many more examples from algebraic geometry and representation theory.
\end{expl}

From the point of view of homological algebra, any Grothendieck category $\mcG$ has an important object: an injective cogenerator $W$ (see~\cite[Th\'eor\`eme 1.10.1]{Gro57} and~\cite[Corollary 2.11]{Mitch64}). One can ask, given a tilting object $T \in \mcG$, which properties characterize the image of $W$ under the triangle equivalence $\RHom_\mcG(T,-)$. The rest of the paper is mostly devoted to giving a satisfactory answer to this question.

In order to formulated our answer, we need the following definition from~\cite{ACo01}. If $M$ is an object in a Grothendieck category, we will denote by $\Prod M$ the class of all summands in products of copies of $M$ and by $\Add M$ the class of all summands in coproducts of copies of $M$. Thus, if $W \in \mcG$ is an injective cogenerator, then $\Prod W$ is precisely the class of all injective objects in $\mcG$. Now we can define:

\begin{defn}[\cite{ACo01}] \label{defn:big-cotilt}
Let $R$ be a ring, $C$ a right $R$-module and $n \ge 0$. Then $C$ is a \emph{big $n$-cotilting module} if it satisfies the following three conditions:
\begin{enumerate}
\item[(BC1)] $C$ has injective dimension bounded by $n$. That is there exists an exact sequence in $\ModR$ of the form
\[ 0 \la C \la E_0 \la E_1 \la \cdots \la E_n \la 0 \]
with all the $E_i$ injective.

\item[(BC2)] $\Ext^j_R(C^I,C) = 0$ for every $j > 0$ and every set $I$. Here, $C^I$ is the product of copies of $C$ indexed by $I$.

\item[(BC3)] There exists an exact sequence
\[ 0 \la C_r \la \dots \la C_1 \la C_0 \la Q \la 0 \]
in $\ModR$ \st $Q$ is an injective cogenerator of $\ModR$, $r \ge 0$ and $C_i \in \Prod C$ for all $i$.
\end{enumerate}

Moreover, $C$ is called \emph{big cotilting} if it is big $n$-cotilting for some $n \ge 0$.
\end{defn}

The adjective \emph{big} refers to the fact that $C$ need not be finitely generated as a module. There is an extensive theory for such modules developed by several authors which covers various homological, model theoretic and approximation properties. We refer to~\cite{GT06} for the results and references. A part of the theory will be recalled and used later in the text. Here we only mention that the sequence in (BC3) can be taken so that $r = n$, where $n$ is the injective dimension of $C$; see~\cite[Proposition 3.5]{Baz04}.

In order not to postpone the main point any further, we shall state the main result of the section:

\begin{thm} \label{thm:cotilt-from-tilt}
Let $\mcG$ be a Grothendieck category, $W \in \mcG$ be an injective cogenerator and $T \in \mcG$ be a classical tilting object. Denote $R = \End_\mcG(T)$. Then all the cohomologies of $\RHom_\mcG(T,W) \in \Der R$ in degrees different from zero vanish, so that $\RHom_\mcG(T,W) \cong C$ in $\Der R$ for some module $C \in \ModR$.

Suppose further that there exists $n \ge 0$ \st $\Ext^{n+1}_\mcG(T,-) \equiv 0$ (one often says that the projective dimension of $T$ in $\mcG$ is less than or equal to $n$). Then $C$ is a big $n$-cotilting module. 
\end{thm}

\begin{rem} \label{rem:why-big}
This result clearly demonstrates why we should deal with \emph{big} cotilting modules. It happens very often that injective objects of Grothendieck categories are not small in any reasonable sense (finitely presentable, compact in the derived category or similar). Thus, we cannot expect $C$ to be a finitely generated or presented module, and this is indeed the case in examples.
\end{rem}

\begin{rem} \label{rem:P1-and-RR}
If $\mcG = \Qco{\PP 1k}$, where $k$ is an algebraically field, and $T = \OO \oplus \OO(1)$, the resulting $k$-algebra $R = \End_\mcG(T)$ is a so-called Kronecker algebra. Finitely generated $R$-modules are well understood (see~\cite[\S VIII.7]{ARS97}) and this was used by Ringel and Reiten to obtain in~\cite[Proposition 10.1]{RR06} a very concrete description of what turns out to be a big cotilting module coming from Theorem~\ref{thm:cotilt-from-tilt} in that particular case.
\end{rem}

\begin{rem} \label{rem:loc-noeth}
If $\mcG$ is a locally noetherian Grothendieck category, such as in the case of Example~\ref{expl:beilinson} or Remark~\ref{rem:P1-and-RR}, then there is an injective cogenerator $W \in \mcG$ \st $\Prod W = \Add W$. It turns out by inspecting the proof of Theorem~\ref{thm:cotilt-from-tilt} that in such a case $\Prod C = \Add C$ in $\ModR$ for the cotilting module $C$. This among others implies that $C$ is a $\Sigma$-pure-injective module by~\cite[Lemma 1.2.23]{GT06} and this is a very restrictive condition on $C$. A partial converse saying that if a $1$-cotilting module $C$ coming from a Grothendieck category $\mcG$ with a classical tilting object is $\Sigma$-pure-injective, then $\mcG$ is locally noetherian, has been proved in~\cite{CMT10}.
\end{rem}

\begin{rem} \label{rem:pdim-of-T}
If $\mcG$ has exact products, then the existence of $n\ge0$ \st $\Ext^{n+1}_\mcG(T,-) \equiv 0$ follows from the compactness of $T$. Indeed, if there were no such $n$, then we would have objects $M_i \in \mcG$ with $\Ext^{i}_\mcG(T,M_i) \ne 0$ for arbitrary high $i > 0$. In particular, there would be a morphisms $T \to \prod_i \Sigma^i M_i$ in $\DerG$ with infinitely many non-zero components. But using the exactness of products, we have $\prod_i \Sigma^i M_i = \coprod_i \Sigma^i M_i$ both in $\CpxG$ and in $\DerG$, contradicting the compactness of~$T$.

We do not know, however, whether $\Ext^{n+1}_\mcG(T,-) \equiv 0$ for some $n$ follows from the compactness if products are not exact. This is in fact closely related to rigidity of $t$-structures defined in~\cite[\S5]{AJS03}. Various Grothendieck categories with non-exact products come from algebraic geometry (see for instance \cite[Example 4.9]{Kr05}), but for these an integer $n$ as above exists in all examples we are aware of.
\end{rem}

Despite being somewhat technical, the proof of Theorem~\ref{thm:cotilt-from-tilt} uses only standard tools. First we identify the copy of $\ModR$ which occurs in $\Der\mcG$ thanks to the equivalence $\RHom_\mcG(T,-)$.

\begin{lem} \label{lem:copy-of-ModR}
Let $\mcG$ be a Grothendieck category and $T \in \mcG$ be a classical tilting object. Given $X \in \Der\mcG$, the following are equivalent:
\begin{enumerate}
\item $\RHom_\mcG(T,X)$ is isomorphic to an $R$-module.
\item $\Hom_{\Der\mcG}(\Sigma^i T,X) = 0$ for all $i \ne 0$.
\end{enumerate}
In particular, $\RHom_\mcG(T,W) \in \ModR$ whenever $W \in \mcG$ is injective.
\end{lem}

\begin{proof}
Clearly $Y \in \Der R$ is isomorphic to a module \iff $\Hom_{\Der R}(\Sigma^iR,Y) = 0$ for all $i \ne 0$. Then we simply use that $\RHom_\mcG(T,-)\dd \Der\mcG \to \Der R$ is a triangle equivalence which sends $T$ to $R$. Regarding the last statement, we use that $\Hom_{\Der\mcG}(\Sigma^i T,W) \cong \Ext^{-i}_\mcG(T,W) = 0$ for all $i \ne 0$.
\end{proof}

Our next goal is to identify the complex corresponding to an injective cogenerator $Q \in \ModR$ in $\Der\mcG$. We will find it using a variant of Brown-Comenetz duality~\cite{BC76} in the spirit of~\cite[Remark 8.5.22 and \S8.7]{Nee01}. Versions of this construction are also known as Serre duality in algebraic geometry (see~\cite{BoKa89}) and Auslander-Reiten translation in representation theory (see~\cite{KrLe06}).

\begin{defn} \label{defn:BC-dual}
Let $\mcT$ be a triangulated category with small coproducts and let $Z \in \mcT$ be a compact object. A \emph{Brown-Comenetz dual} of $Z$ in $\mcT$ is an object $\BC Z \in \mcT$ \st there exists a natural equivalence
\[ \Hom_\mcT\big(-,\BC Z\big) \mapr{\cong} \Hom_\bbZ\big(\Hom_\mcT(Z,-), \bbQ/\bbZ\big) \]
of functors $\mcT\op \to \Ab$.
\end{defn}

Note that if $\BC Z$ exists, it is determined uniquely up to isomorphism by the Yoneda lemma. A sufficient condition for the existence is given by the following proposition.

\begin{prop} \label{prop:BC-dual}
Let $\mcT$ be a triangulated category with small coproducts and a compact generator in the sense of Definition~\ref{defn:comp-gen}. Then every compact object of $\mcT$ has a Brown-Comenetz dual.
\end{prop}

\begin{proof}
Given a compact object $Z \in \mcT$, the functor
\[ \Hom_\bbZ\big(\Hom_\mcT(Z,-), \bbQ/\bbZ\big)\dd \mcT\op \la \Ab \]
turns triangles to long exact sequences and preserves products. Under the condition on~$\mcT$, the functor is representable by an object of $\mcT$ by~\cite[Theorem 8.3.3]{Nee01}. The representing object is by the very definition $\BC Z$, the Brown-Comenetz dual of~$Z$.
\end{proof}

The above is already sufficient to construct a preimage of an injective cogenerator of $\ModR$ directly in $\Der\mcG$.

\begin{cor} \label{cor:inj-cogen-in-D(G)}
Let $\mcG$ be a Grothendieck category with a classical tilting object $T$. Let $R = \End_\mcG(T)$ and denote $Q = \Hom_\bbZ(R,\bbQ/\bbZ)$, viewed as a right $R$-module. Then $Q$ is an injective cogenerator of $\ModR$ and $\RHom_\mcG(T,\BC T) \cong Q$.
\end{cor}

\begin{proof}
That $Q \in \ModR$ is an injective cogenerator is standard, see~\cite[Lemma 18.5]{AF92}. In order to prove that $\RHom_\mcG(T,\BC T) \cong Q$, it suffices to prove that the Brown-Comenetz dual of $R$ in $\Der R$ is isomorphic to $Q$. This follows by inspecting the degree zero cohomology in the composition of the natural isomorphisms
\begin{multline*}
\RHom_R(-,Q) \cong \RHom_R\big(-,\RHom_\bbZ(R,\bbQ/\bbZ)\big) \cong \\
\cong \RHom_\bbZ(- \Lotimes_R R,\bbQ/\bbZ) \cong \RHom_\bbZ(\RHom_R(R,-),\bbQ/\bbZ).
\qedhere
\end{multline*}
\end{proof}

We shall recall another definition.

\begin{defn} \label{defn:susp}
Let $\mcT$ be a triangulated category with small coproducts. A full subcategory $\mcS \subseteq \mcT$ is called \emph{suspended} if $\mcS$ is closed under taking coproducts, suspensions, and if $X \to Y \to Z \to \Sigma X$ is a triangle with $X,Z \in \mcS$, then $Y \in \mcS$.

Dually, if $\mcT$ is triangulated with small products, $\mcS \subseteq \mcT$ is \emph{cosuspended} if $\mcS$ is closed under products, desuspensions, and if $\Sigma\inv Z \to X \to Y \to Z$ is a triangle with $X,Z \in \mcS$, then $Y \in \mcS$.
\end{defn}

The defining property of $\BC T$ allows us to prove a relation between when cohomologies of $X$ and cohomologies of $\RHom_\mcG(T,X)$ vanish. What we prove among others implies that a quasi-inverse of the functor $\RHom_\mcG(T,-)\dd \DerG \to \Der R$ is way-out right in the sense of~\cite[\S I.7]{Hart66}.

\begin{lem} \label{lem:way-out}
Let $\mcG$ be a Grothendieck category, $T$ be a classical tilting object, and suppose that there exists $n \ge 0$ \st $\Ext^{n+1}_\mcG(T,-) \equiv 0$. Then:
\begin{enumerate}
\item $\BC T$ can be, up to isomorphism in $\DerG$, represented by a complex whose all components are injective and which has non-zero components only in cohomological degrees $(-n)$ to $0$.
\item If $X \in \DerG$ is \st $\Hom_\DerG(\Sigma^i T, X)=0$ for all $i>0$, then the cohomology objects of $X$ satisfy $H^i_\mcG(X) = 0$ for all $i < -n$.
\end{enumerate}
\end{lem}

\begin{proof}
(1) Using dimension shifting in the second argument of Ext, one easily shows that $\Ext^j_\mcG(T,-) \equiv 0$ for all $j > n$. This and the definition of the Brown-Comenetz dual implies that
\[ \Hom_\DerG\big(\Sigma^i Z, \BC T\big) = 0 \quad \textrm{for all $Z \in \mcG$ and $i < 0$ or $i > n$.}  \]
Now using either~\cite[Theorem 5.4]{AJS00} or the injective model structure discussed later in Example~\ref{expl:inj-model}, we can replace $\BC T$ by an isomorphic object $K' \in \DerG$ \st all components of $K'$ are injective in $\mcG$ and
\[ \Hom_\DerG\big(Z, \BC T\big) \cong \Hom_\HtpG(Z, K') \quad \textrm{for all $Z \in \DerG$.}  \]
Let us denote by $Z^i(K')$ the $i$-th cocycle object of $K'$. By combining the above, we observe that the canonical cochain complex maps $\Sigma^{-i}Z^i(K') \to K'$ vanish in $\HtpG$ for $i<-n$ and for $i>0$. In other words, for each such $i$ we have a map $s^i\dd Z^i(K') \to (K')^{i-1}$ such that $\dif^{i-1}_{K'} \circ s^i$ is equal to the inclusion $Z^i(K') \to (K')^i$. In particular, $H^i(K') = 0$ and the epimorphism $(K')^{i-1} \to \Img \dif^{i-1}_{K'} = Z^i(K')$ splits. Let us denote by $K \in \CpxG$ the complex concentrated in cohomological degrees $(-n)$ to $0$ given by
\[
\cdots \to 0 \to (K')^{-n}/\Img\dif^{-n-1}_{K'} \to (K')^{-n+1} \to \cdots \to (K')^{-1} \to Z^0(K') \to 0 \to \cdots
\]
A simple computation reveals that $K'$ is isomorphic in $\HtpG$ to $K$ and $K$ has injective components. Indeed, let us denote by $K''$ the subcomplex of $K'$ given by
\[
(K'')^i = \begin{cases}
(K')^i  & i < 0, \\
Z^0(K') & i = 0, \\
0       & i > 0.
\end{cases}
\]
That is, $K''$ is a truncation of $K'$ to non-positive cohomological degrees. By the above computation we have $(K')^0 = \Img s^1 \oplus (K'')^0$ and so we have a component wise split short exact sequence $0 \to K'' \to K' \to K'/K'' \to 0$ in $\Cpx{R}$ with $K'/K''$ acyclic. In particular, $K''$ has injective components and is quasi-isomorphic to $K'$. Similarly we can express $K$ as a component wise split quasi-isomorphic factor of $K''$. Thus, $K$ is isomorphic to $\BC T$ in $\Der{R}$ and has the required properties.

(2) Note that both $\mcS_1 = \{X \in \DerG \mid \Hom_\DerG(\Sigma^i T, X)=0 \textrm{ for all } i>0 \}$ and $\mcS_2 = \{X \in \DerG \mid H^i_\mcG(X) = 0 \textrm{ for all } i < -n \}$ are cosuspended subcategories of $\DerG$. For $\mcS_2$ this follows from the fact that every object of $\mcS_2$ is isomorphic in $\DerG$ to a complex with all components injective and non-zero components only in cohomological degrees $\ge (-n)$. Our task is to prove that $\mcS_1 \subseteq \mcS_2$.

Using the same argument as for Lemma~\ref{lem:copy-of-ModR}, the essential image of $\mcS_1$ in $\Der R$ is the cosuspended subcategory
\[ \mcU = \{Y \in \Der R \mid H^i_R(Y) = 0 \textrm{ for all } i < 0 \}. \]
We claim that $\mcU$ is the smallest suspended subcategory of $\Der R$ containing $Q = \Hom_\bbZ(R,\bbQ/\bbZ)$. Indeed, $\mcU$ clearly contains all bounded complexes of injective modules concentrated in cohomological degrees $\ge 0$. If $Y \in \mcU$ is arbitrary, we can without loss of generality assume that $Y$ is a complex of injective modules concentrated in cohomological degrees $\ge 0$. Then $Y$ is a homotopy limit in the sense of~\cite[Remark 2.3]{BN93} of its bounded brutal truncations. This proves the claim.

Passing back through the triangle equivalence and using Corollary~\ref{cor:inj-cogen-in-D(G)}, we learn that $\mcS_1$ is the smallest cosuspended subcategory of $\DerG$ containing $\BC T$. As $\mcS_2$ contains $\BC T$ by (1), we have $\mcS_1 \subseteq \mcS_2$.
\end{proof}

Now we can finish the proof of the theorem.

\begin{proof}[Proof of Theorem~\ref{thm:cotilt-from-tilt}]
Suppose that $T \in \mcG$ is a classical tilting object and $W \in \mcG$ is an injective cogenerator. Then $\RHom_\mcG(T,W)$ is isomorphic to a module by Lemma~\ref{lem:copy-of-ModR}. We denote this module by $C$.

Suppose now that we have a non-negative integer $n$ \st $\Ext^{n+1}_\mcG(T,-) \equiv 0$. Thus, whenever $X \in \DerG$ \st $Y \cong \RHom_\mcG(T,X)$ for a module $Y \in \ModR$, it follows from Lemmas~\ref{lem:copy-of-ModR} and~\ref{lem:way-out}(2) that
\[ \Ext^{n+1}_R(Y,C) \cong \Hom_{\Der R}(Y,\Sigma^{n+1} C) \cong \Hom_\DerG(X,\Sigma^{n+1} W) = 0. \]
In particular $\Ext_R^{n+1}(-,C) \equiv 0$ in $\ModR$ and (BC1) of Definition~\ref{defn:big-cotilt} holds for $C$.

Regarding (BC2), note that given any set $I$ and the product $W^I$ in $\mcG$, then $W^I$ is injective in $\mcG$ and
\[ \Hom_\DerG(X,W^I) \cong \Hom_\HtpG(X,W^I) \quad \textrm{for all $X \in \DerG$}. \]

Therefore, although the canonical inclusion $\mcG \to \DerG$ does not preserve products in general, $W^I$ \emph{is} the corresponding product of copies of $W$ in $\DerG$. It follows that
\[ \Ext^i_R(C^I,C) \cong \Hom_{\Der R}(\Sigma^{-i} C^I,C) \cong \Hom_\DerG(\Sigma^{-i}W^I,W) \cong \Ext^i_\mcG(W^I,W) = 0 \]
for all $i>0$, as required.

Finally, let $K$ be a complex isomorphic to $\BC T$ in $\Der\mcG$ as in Lemma~\ref{lem:way-out}(1). Hence we have an isomorphism $Q \cong \RHom_\mcG(T,K)$ in $\Der R$, where $Q$ is the injective cogenerator of $\ModR$ as in Corollary~\ref{cor:inj-cogen-in-D(G)}. Moreover, $\RHom_\mcG(T,K)$ can be taken of the form
\[ \cdots \la 0 \la 0 \la C^{-n} \la \cdots \la C^{-1} \la C^0 \la 0 \la 0 \la \cdots \]
where each $C^i$ is in cohomological degree $i$ and belongs to $\Prod C$. As this complex has a non-zero cohomology only in degree zero and the cohomology $R$-module is isomorphic to $Q$, (BC3) follows.
\end{proof}

The aim of the rest of the text is to prove that \emph{every} big cotilting module arises as in Theorem~\ref{thm:cotilt-from-tilt}. For big $1$-cotilting modules this was essentially proved by Colpi, Gregorio and Mantese in~\cite{CGM07}. If $n > 1$, we need more powerful tools in order to construct a corresponding Grothendieck category with a classical tilting object.

\section{Resolving subcategories and resolution dimension}
\label{sec:resolving}

In this section we generalize some results of Auslander and Bridger~\cite{AuBr69}. We will be concerned with the following concept.

\begin{defn} \label{defn:resolving}
Let $\mcA$ be an abelian category and $\mcX \subseteq \mcA$ be a full subcategory. Then $\mcX$ is a \emph{resolving subcategory} if it satisfies the following.
\begin{enumerate}
\item[(R0)] $\mcX$ is closed under retracts.
\item[(R1)] Suppose that $0 \to X \to Y \to Z \to 0$ is a short exact sequence and $Z \in \mcX$. Then $X \in \mcX$ \iff $Y \in \mcX$.
\item[(R2)] $\mcX$ is generating. That is, for any $A \in \mcA$ there exists an epimorphism $X \to A$ with $X \in \mcX$. 
\end{enumerate}

A resolving subcategory is called \emph{functorially resolving} provided that the epimorphisms in (R2) can be taken functorially in $A$.

The \emph{resolution dimension} of an object $A \in \mcA$ with respect to $\mcX$, denoted $\resdim_\mcX A$, is defined as the smallest integer $n\ge0$ \st there exists an exact sequence
\[ 0 \la X_n \la \cdots \la X_1 \la X_0 \la A \la 0 \]
with $X_0,X_1,\dots,X_n \in \mcX$, or $\resdim_\mcX A = \infty$ if such $n$ does not exist.

\emph{Coresolving} and \emph{functorially coresolving} subcategories and \emph{coresolution dimension} are defined dually.
\end{defn}

In Grothendieck categories, the functorial property of resolving subcategories comes often for free.

\begin{lem} \label{lem:functor-for-free}
Let $\mcG$ be a Grothendieck category. If $\mcX$ is a resolving subcategory closed under small coproducts, then $\mcX$ is functorially resolving. Dually, if $\mcX$ is coresolving closed under products, then $\mcX$ is functorially coresolving.
\end{lem}

\begin{proof}
A resolving subcategory in $\mcG$ necessarily contains a generator $G$ for $\mcG$. Then we can choose for every $A \in \mcG$ the canonical epimorphism
\[ G^{(\Hom_\mcG(G,A))} \la A \]
and it is an easy observation that this choice can be made functorial in $A$. Similarly, every coresolving subcategory contains an injective cogenerator.
\end{proof}

The main observation of the section is that the resolution dimension \wrt $\mcX$ shares some nice properties with usual projective dimension. In particular, the resolution dimension of an object can be computed using \emph{any} resolution by objects of $\mcX$. 

\begin{prop} \label{prop:res-dim}
Let $\mcA$ be an abelian category, $\mcX \subseteq \mcA$ a resolving subcategory, and $n \ge 0$ an integer. Then the following hold:
\begin{enumerate}
\item Suppose that $A \in \mcA$ is an object and that we have exact sequences
\[ 0 \la K_n \la X_{n-1} \la \cdots \la X_1 \la X_0 \mapr{p} A \la 0 \]
and
\[ 0 \la K'_n \la X'_{n-1} \la \cdots \la X'_1 \la X'_0 \mapr{p'} A \la 0 \]
\st $X_0, \dots, X_{n-1}, X'_0, \dots, X'_{n-1} \in \mcX$. Then $K_n \in \mcX$ \iff $K'_n \in \mcX$.

\item The class $\mcX_n$ of all objects of $\mcA$ of resolution dimension $\le n$ is a resolving subcategory.
\end{enumerate}

The dual statements hold for coresolving subcategories and coresolution dimension.
\end{prop}

\begin{proof}
The first part was proved in~\cite[Lemma 3.12]{AuBr69} for abelian categories with enough projectives. In general we will prove (1) and (2) jointly by induction on $n$.

If $n=0$, then both (1) and (2) are clear. Suppose that $n > 0$. Regarding (1), it clearly suffices to prove that $K'_n \in \mcX$ provided that $K_n \in \mcX$. Hence suppose $K_n \in \mcX$ and consider the pullback diagram
\[
\begin{CD}
   @.             @.      0     @.      0             \\
@.          @.          @VVV          @VVV            \\
   @.             @.   \Ker p   @=    \Ker p          \\
@.          @.          @VVV          @VVV            \\
0  @>>>  \Ker p'  @>>>    Z     @>>>   X_0  @>>>  0   \\
@.          @|          @VVV         @VV{p}V          \\
0  @>>>  \Ker p'  @>>>  X'_0  @>{p'}>>  A   @>>>  0   \\
@.          @.          @VVV          @VVV            \\
   @.             @.      0     @.      0
\end{CD}
\]
We have $X_0 \in \mcX$ and $\Ker p \in \mcX_{n-1}$. The inductive hypothesis on $\mcX_{n-1}$ implies that $Z \in \mcX_{n-1}$ and in turn $\Ker p' \in \mcX_{n-1}$. Applying (1) for $n-1$ gives us $K'_n \in \mcX$, as required.

For (2) suppose that we have a short exact sequence $0 \to A \overset{i}\to B \overset{p}\to C \to 0$ in $\mcA$. Let us fix epimorphisms $f\dd X_1 \to A$ and $g\dd X_2 \to B$ \st $X_1, X_2 \in \mcX$. The we can construct the diagram with exact rows and columns:
\[
\begin{CD}
    @.     0              @.                0           @.            0              \\
@.       @VVV                             @VVV                      @VVV             \\
0   @>>>   K              @>>>              L           @>>>          M   @>>>   0   \\
@.       @VVV                             @VVV                      @VVV             \\
0   @>>>  X_1         @>{(^1_0)}>>  X_1 \oplus X_2  @>{(0\;1)}>>     X_2  @>>>   0   \\
@.     @V{f}VV                       @V{(if\;g)}VV                @V{pg}VV           \\
0   @>>>   A              @>{i}>>           B           @>{p}>>       C   @>>>   0   \\
@.       @VVV                             @VVV                      @VVV             \\
    @.     0              @.          0                 @.            0
\end{CD}
\]
If $A,C \in \mcX_n$, then $K,M \in \mcX_{n-1}$. Using the inductive hypothesis, we have $L \in \mcX_{n-1}$ and so $B \in \mcX_n$. Similarly if $B,C \in \mcX_n$, then $A \in \mcX_n$. This proves Definition~\ref{defn:resolving}(R1) for $\mcX_n$. Moreover, if $p$ splits, then the diagram can be constructed so that the upper row also splits. In that case if $B \in \mcX_n$, then $L \in \mcX_{n-1}$, so $K,M \in \mcX_{n-1}$ and $A,C \in \mcX_n$. Thus, (R0) holds. Finally, the inclusion $\mcX \subseteq \mcX_n$ implies (R2) and $\mcX_n$ is resolving.
\end{proof}

\section{Cotilting model structures}
\label{sec:models}

This section is devoted to a description of suitable model structures (in the sense of \cite{Hir03,Hov99}) on the category $\Cpx\mcG$ for a Grothendieck category $\mcG$ \st the homotopy category is $\Der\mcG$. We will also exhibit a Quillen equivalence between $\Cpx R$ and $\Cpx\mcG$ if $\mcG$ is a Grothendieck category with a classical tilting object and $R$ is its endomorphism ring.

\subsection{Models for the derived category}
\label{subsec:model-derived}

The derived category of $\mcG$ is usually defined as the category which we obtain from $\CpxG$ by formally inverting all quasi-isomorphisms. This process of inverting morphisms fits into a well understood framework of Quillen model categories~\cite{Hir03,Hov99,QHtp67} from homotopy theory. Although in principle one can always formally construct $\mcA[\we\inv]$ for a category $\mcA$ with a distinguished class $\we$ of morphisms (see~\cite{GZ67}), it is not so easy to compute in $\mcA[\we\inv]$ and also a set-theoretic difficulty may occur. Namely, the collections of morphisms $\mcA[\we\inv](X,Y)$ may not be small in general and we may be forced to pass to a higher universe to make $\mcA[\we\inv]$ a well defined category. If $(\mcA,\we)$ is a part of a model structure, this problem does not occur and the description of $\mcA[\we\inv]$ as a subquotient of $\mcA$ (see~\cite[Theorem 1.2.10]{Hov99} or~\cite[\S8.4]{Hir03}) is much more accessible to computation.

Recall that a \emph{model structure} on $\mcA$ consists of the class $\we$ of \emph{weak equivalences}, i.e.~the morphisms which we wish to make invertible, and of two additional classes of morphisms $\cof,\fib$, called \emph{cofibrations} and \emph{fibrations}, respectively. These three classes are subject to well-known axioms for which we refer to~\cite{Hir03,Hov99}. A \emph{model category} is a complete and cocomplete category equipped with a model structure $(\cof,\we,\fib)$, and in this context $\mcA[\we\inv]$ is traditionally called the \emph{homotopy category} of $\mcA$.

As we are mostly interested in $\DerG$ for a Grothendieck category $\mcG$, we shall start with $\mcA = \CpxG$ and $\we$ the class of all quasi-isomorphisms. The question now stands as how to construct the additional classes $\cof,\fib$, preferably in a way which later allows for efficient computations in $\Der R$. Here we use the fact that $\CpxG$ is not just an arbitrary category. It is an abelian category, and even a Grothendieck category (see for instance~\cite[Lemma 1.1]{St13-deconstr}). It is very helpful to make the model structure respect the abelian structure the sense defined by Hovey~\cite{Hov02}:

An \emph{abelian model structure} on an abelian category $\mcA$ is a model structure \st cofibrations are precisely monomorphisms with cofibrant cokernels and fibrations are precisely epimorphisms with fibrant kernels. Of course, an \emph{abelian model category} is a complete and cocomplete abelian category together with an abelian model structure.

In fact, Hovey~\cite{Hov02} did much more than just giving a suitable compatibility condition on model and abelian structures. He showed that in this situation, the model structure is fully given by classes of objects rather than morphisms, and the definition of model structure translates to the existence certain cotorsion pairs. Methods for obtaining such model structures on $\Cpx R$ were initially studied by Gillespie~\cite{Gil04} and since then several articles treating the topic have been published. We refer to~\cite{St13-ICRA} for a detailed treatment and references. Here we recall only the necessary minimum: the concept of a cotorsion pair (originally due to Salce~\cite{Sal79}), Hovey's translation and a generalized version of Gillespie's result.

\begin{defn} \label{defn:cotorsion}
Let $\mcA$ be an abelian category. A \emph{complete cotorsion pair} is a pair of classes of objects $(\mcC,\mcF)$ \st
\begin{enumerate}
\item[(CP0)] $\mcC$ and $\mcF$ are closed under retracts.
\item[(CP1)] $\Ext^1_\mcA(C,F) = 0$ for all $C \in \mcC$ and $F \in \mcF$.
\item[(CP2)] For every object $X \in \mcA$ there exist (non-unique) short exact sequences
\[
0 \la X \la F_X \la C_X \la 0 \quad \textrm{ and } \quad
0 \la F^X \la C^X \la X \la 0
\]
with $C_X, C^X \in \mcC$ and $F_X, F^X \in \mcF$.
\end{enumerate}

The sequences as in (CP2) are called \emph{approximation sequences}. The complete cotorsion pair is called \emph{functorially complete} if the approximation short exact sequences can be chosen functorially in $X$.

The complete cotorsion pair is called \emph{hereditary} if $\Ext^i_\mcA(C,F) = 0$ for all $C \in \mcC$, $F \in \mcF$ and $i > 0$.
\end{defn}

For future reference, we shall also explain the motivation for the term approximation sequence. The idea is that we approximate a general object $X \in \mcA$ by an object from a specified full subcategory of $\mcA$. We have two versions of these approximations---precovers (right approximations) and preenvelopes (left approximations); see~\cite{ARS97,GT06}.

\begin{defn} \label{defn:approx}
Let $\mcA$ be a category and $\mcC \subseteq \mcA$ be a full subcategory. A morphism $f\dd C^X \to X$ is a \emph{$\mcC$-precover} if $C^X \in \mcC$ and any other morphism $f'\dd C' \to X$ with $C' \in \mcC$ factors through $f$. Dually, $g\dd X \to C_X$ is a \emph{$\mcC$-preenvelope} if $C_X \in \mcC$ and any other morphism $g'\dd X \to C''$ with $C'' \in \mcC$ factors through $g$.
\end{defn}

\begin{lem} \label{lem:approx}
Let $(\mcC,\mcF)$ be a complete cotorsion pair in an abelian category $\mcA$. Given the approximation sequences for $X \in \mcA$ as in Definition~\ref{defn:cotorsion}, then $C^X \to X$ is a $\mcC$-precover and $X \to F_X$ is an $\mcF$-preenvelope.
\end{lem}

\begin{proof}
Given $C' \in \mcC$, then the homomorphism $\Hom_\mcA(C',C^X) \to \Hom_\mcA(C',X)$ is an epimorphism since $\Ext^1_\mcA(C',F^X) = 0$. The other case is dual.
\end{proof}

Now we are in a position to state the crucial result due to Hovey relating abelian model structures and cotorsion pairs.

\begin{prop}[{\cite[Theorem 2.2]{Hov02}}] \label{prop:abel-model}
Let $\mcA$ be an abelian model category and $\Cof,\We,\Fib$ be the classes of cofibrant, trivial and fibrant objects, respectively. Then the following hold:
\begin{enumerate}
\item The class $\We$ is closed under retracts. Moreover, if $0 \to W_1 \to W_2 \to W_3 \to 0$ is a short exact sequence and two of $W_1,W_2,W_3$ belong to $\We$, so the third.
\item The pairs $(\Cof\cap\We,\Fib)$ and $(\Cof,\We\cap\Fib)$ are functorially complete cotorsion pairs.
\end{enumerate}
Moreover, every triple $(\Cof,\We,\Fib)$ of classes of objects of $\mcA$ which satisfies conditions (1) and (2) uniquely determines an abelian model structure on $\mcA$
\end{prop} 

Thus, our task now stands as follows. We start with $\mcA = \CpxG$ and the class $\We$ of all acyclic complexes, and we need to obtain $\Cof,\Fib$ \st $(\Cof\cap\We,\Fib)$ and $(\Cof,\We\cap\Fib)$ are functorially complete cotorsion pairs. A result originally due to Gillespie gives such a construction when we start with a suitable single complete cotorsion pair $(\Cof_0,\Fib_0)$ in $\mcG$. First of all, $(\Cof_0,\Fib_0)$ needs to be what a topologist would call cofibrantly generated:

\begin{prop} \label{prop:Quillen-cotorsion}
Let $\mcG$ be a Grothendieck category and $\mcS$ be a small set of objects containing a generator for $\mcG$. Then there is a functorially complete cotorsion pair $(\mcC,\mcF)$ \st
\begin{enumerate}
\item $\mcF = \{ F \in \mcG \mid \Ext^1_\mcG(S,F) = 0 \textrm{ for all } S \in \mcS \}$.
\item The class $\mcC$ consists precisely of retracts of objects $X$ for which there exists a direct system of monomorphisms $(X_\alpha \mid \alpha\le\lambda)$ which is indexed by an ordinal number $\lambda$ and which satisfies the following conditions:

  \begin{enumerate}
    \item $X_0 = 0$ and $X_\lambda = X$;
    \item $X_\beta \cong \li_{\alpha<\beta} X_\alpha$ canonically for each limit ordinal $\beta$;
    \item $\Coker (X_\alpha \to X_{\alpha+1})$ is isomorphic to an object of $\mcS$ for each $\alpha<\lambda$.
  \end{enumerate}
\end{enumerate}
\end{prop}

\begin{proof}
The proof uses Quillen's small object argument. We refer either to~\cite[Corollary 2.15(2)]{SaoSt11} or to~\cite[Theorem 5.16]{St13-ICRA}. The ideas can be traced back to~\cite{ET01} for module categories, and to~\cite{Gil07,Hov02,Ro02} in the general case.
\end{proof}

\begin{defn} \label{defn:cotorsion-set}
A (functorially) complete cotorsion pair in a Grothendieck category which occurs as in Proposition~\ref{prop:Quillen-cotorsion} is said to be \emph{generated by a small set}.
\end{defn}

Note that Proposition~\ref{prop:Quillen-cotorsion}(2) says that $\mcC$ consists of summands of transfinitely iterated extensions of objects of $\mcS$.
Now we can state the main existence result for model structures. In order to facilitate the statement, we shall use the following notation from~\cite{Gil04}:

\begin{nota} \label{nota:Gil-tilde}
If $\mcE \subseteq \mcG$ is a class of objects closed under extensions, then $\tilde\mcE \subseteq \CpxG$ will stand for the class of all acyclic complexes with all cocycle objects in $\mcE$.
\end{nota}

\begin{prop} \label{prop:models}
Let $\mcG$ be a Grothendieck category and denote by $\we$ the class of all quasi-isomorphisms in $\CpxG$. If $(\mcC_0,\mcF_0)$ is a complete hereditary cotorsion pair in $\mcG$ which is generated by a small set, then there exists an abelian model structure $(\cof,\we,\fib)$ on $\CpxG$ described as follows:
\begin{enumerate}
\item Cofibrations are the monomorphisms $f$ \st $\Ext^1_\CpxG(\Coker f,X) = 0$ for all complexes $X \in \tilde\mcF_0$.
\item Trivial cofibrations are the monomorphisms $f$ \st $\Coker f \in \tilde\mcC_0$.
\item Fibrations are the epimorphisms $g$ \st $\Ext^1_\CpxG(X,\Ker g) = 0$ for all complexes $X \in \tilde\mcC_0$.
\item Trivial fibrations are the epimorphisms $g$ \st $\Ker g \in \tilde\mcF_0$.
\end{enumerate}
The homotopy category of this model category is $\DerG$.
\end{prop}

\begin{proof}
We refer to~\cite[Theorem 4.22]{SaoSt11} or~\cite[Theorem 7.16]{St13-ICRA}. The statement was originally obtained in a more restricted form by Gillespie~\cite{Gil04} and appears in various forms in the literature in~\cite{EJv2,EGPT12,Gil07,Gil11,YL11} and most certainly also elsewhere.
\end{proof}

\begin{rem} \label{rem:desc-cofib}
The description of general fibrations and cofibrations, as opposed to their trivial versions, is not very straightforward. It may be difficult to decide in general whether the groups $\Ext^1_\CpxG(\Coker f,X)$ or $\Ext^1_\CpxG(X,\Ker g)$ vanish. An alternative description of the cokernel of a cofibration which may provide more information was given in~\cite[Proposition 4.5]{St13-deconstr}.

Often, however, we only need the following simple consequence of the adjunctions in~\cite[Lemma 3.1(5) and (6)]{Gil04}: If $f$ is a cofibration, then all components of $\Coker f$ belong to $\mcC_0$. Similarly all components of $\Ker g$ are in $\mcF_0$ for every fibration~$g$.
\end{rem}

An important corollary allows us to compute in $\DerG$:

\begin{cor} \label{cor:compute-in-DerG}
Suppose that $\mcG$ and $(\mcC_0,\mcF_0)$ are as in Proposition~\ref{prop:models}. If $X$ is cofibrant and $Y$ is fibrant, then the canonical homomorphism
\[ \Hom_\HtpG(X,Y) \la \Hom_\DerG(X,Y) \]
is an isomorphism.
\end{cor}

\begin{proof}
This is immediate as soon as one proves that the abstract homotopy relation from the model structure is just the usual cochain complex homotopy. The latter quickly follows from \cite[Proposition 4.4]{Gil11} or~\cite[Proposition 1.1.14]{Beck12}.
\end{proof}

\begin{expl} \label{expl:inj-model}
In any Grothendieck category $\mcG$ we have the trivial functorially complete hereditary cotorsion pair $(\mcG,\mcI)$, where $\mcI$ stands for the class of all injective objects. This easily follows from Proposition~\ref{prop:Quillen-cotorsion} and a version of Baer's Lemma for Grothendieck categories.

Hence we have an abelian model structure in $\CpxG$ in which every object is cofibrant, and fibrant objects are certain complexes with injective components. This folklore model structure is called the \emph{injective model structure} on $\Cpx\mcG$ and it was mentioned in the proof of Lemma~\ref{lem:way-out}.
\end{expl}

\begin{expl} \label{expl:proj-model}
Dually, if $\mcG = \ModR$ for a ring $R$, we have a functorially complete hereditary cotorsion pair $(\mcP,\ModR)$, where $\mcP$ is the class of all projective modules. This gives rise to another folklore model structure, the \emph{projective model structure} on $\Cpx R$; see~\cite[\S2.3]{Hov99}.
\end{expl}

\subsection{A tilting Quillen equivalence}
\label{subsec:QE}

Let us turn back to the case of the derived equivalence $\RHom_\mcG(T,-)\dd \DerG \to \Der R$ from Proposition~\ref{prop:tilt-char} and Theorem~\ref{thm:cotilt-from-tilt}. With respect to suitable model structures we prove that the underlying functor $\Hom_\mcG(T,-)\dd \Cpx\mcG \to \Cpx R$ is a right Quillen equivalence. That is, it not only induces a derived equivalence, but it is also the right hand side of an adjunction which is compatible with the model structures (see~\cite{Hir03,Hov99} for the precise definition). From the point of view of homotopy theory, this says that we have as favorable enrichment for the triangle equivalence as we can hope for.

\begin{thm} \label{thm:QE}
Let $\mcG$ be a Grothendieck category with a classical tilting object $T$ and denote $R = \End_\mcG(T)$. Consider $\Cpx\mcG$ with the injective model structure (Example~\ref{expl:inj-model}) and $\Cpx R$ with the projective model structure (Example~\ref{expl:proj-model}). Then $\Hom_\mcG(T,-)\dd \Cpx\mcG \to \Cpx R$ is a right Quillen equivalence.
\end{thm}

\begin{proof}
The functor $U = \Hom_\mcG(T,-)$ has a left adjoint $F\dd \Cpx R \to \Cpx\mcG$ by the special adjoint functor theorem, \cite[\S V.8]{McL2}. A more explicit description of $F$ will be given later in Remark~\ref{rem:left-adj-to-Hom(T,-)}. At the moment, we only note that both $U$ and $F$ are clearly additive. As a fibration in $\Cpx G$ is an epimorphism whose kernel is a complex with injective components, every component must split and it is easy to show that $U$ preserves fibrations and trivial fibrations. In particular, $(F,U)$ is a Quillen adjunction. Since $\Rder U = \RHom_\mcG(T,-)$ is an equivalence thanks to Proposition~\ref{prop:tilt-char}, $\Hom_\mcG(T,-)\dd \Cpx\mcG \to \Cpx R$ is a right Quillen equivalence as stated.
\end{proof}

\subsection{Cotilting model structures}
\label{subsec:model-cotilting}

Now we shall get to the main point: model structures on $\Cpx R$ attached to a given big cotilting module. First of all, with every cotilting module we can associate a complete cotorsion pair.

\begin{defn} \label{defn:cotilt-class}
Let $C$ be a big cotilting module over a ring $R$. Then the \emph{cotilting class} associated with $C$ is defined as
\[ \mcX_C = \{ X \in \ModR \mid \Ext^i_R(X,C) = 0 \textrm{ for all } i > 0 \}. \]
\end{defn}

\begin{prop} \label{prop:cotilt-cot-pair}
Let $C$ be a big cotilting module over a ring $R$, and put
\[ \mcF_0 = \{ F \in \ModR \mid \Ext^1(X,F) = 0 \textrm{ for all } X \in \mcX_C \}. \]
Then $(\mcX_C,\mcF_0)$ is a functorially complete hereditary cotorsion pair in $\ModR$ which is generated by a small set. Moreover, $\mcX_C$ is functorially resolving, it is closed in $\ModR$ under taking products and direct limits, and we have $\mcX_C \cap \mcF_0 = \Prod C$.
\end{prop}

\begin{proof}
We refer to~\cite[Lemma 8.1.4(c)]{GT06} and to~\cite[Theorem 8.1.7 and 8.1.9]{GT06}. That $\mcX_C$ is functorially resolving follows from Lemma~\ref{lem:functor-for-free}.
\end{proof}

Hence, by virtue of Proposition~\ref{prop:models}, associated with every big cotilting module we have an abelian model structure on $\Cpx R$ whose homotopy category is $\Der R$. In the last few paragraphs of the section, we shall obtain a more satisfactory description of cofibrations. First we recall another piece of standard notation based on~\cite[\S IV.4]{CE56}.

\begin{nota} \label{nota:total-hom}
Suppose that $X,Y \in \Cpx\mcA$, where $\mcA$ is an abelian category. Then the \emph{total Hom-complex} $\HOM_\mcA(X,Y)$ will stand for the complex of abelian groups \st
\[ \HOM_\mcA(X,Y)^n = \prod_{i \in \bbZ} \Hom_\mcA(X^i,Y^{i+n}) \]
for each $i \in \bbZ$, and the differential is defined as the graded commutator. That is, if $f = (f^i)_{i \in \bbZ} \in \prod_{i \in \bbZ} \Hom_\mcA(X^i,Y^{i+n})$ is a collection of morphisms in $\mcA$, we put
\[ \dif_{\HOM_\mcA(X,Y)}(f) = \big( \dif_Y \circ f^i - (-1)^n f^{i+1} \circ \dif_X \big)_{i \in \bbZ}. \]
If $\mcA = \ModR$ for a ring $R$, we shall write $\HOM_R(X,Y)$. 
\end{nota}

It is a standard fact that $H^i\big(\HOM_\mcA(X,Y)\big) \cong \Hom_{\Htp\mcA}(X,\Sigma^i Y)$ for all $i \in \bbZ$. Now we can prove the following result.

\begin{thm} \label{thm:cotilt-models}
Let $R$ be a ring, $C \in \ModR$ be a big cotilting module and $(\mcX_C,\mcF_0)$ the complete cotorsion pair from Proposition~\ref{prop:cotilt-cot-pair}. Then there exists an abelian model structure $(\cof,\we,\fib)$ on $\Cpx R$ described as follows:
\begin{enumerate}
\item Cofibrations are the monomorphisms $f$ \st $\Coker f \in \Cpx{\mcX_C}$.
\item Trivial cofibrations are the monomorphisms $f$ \st $\Coker f \in \tilde\mcX_C$.
\item Fibrations are the epimorphisms $g$ with $\Ext^1_{\Cpx R}(X,\Ker g) = 0$ for $X \in \tilde\mcX_C$.
\item Trivial fibrations are the epimorphisms $g$ \st $\Ker g \in \tilde\mcF_0$.
\end{enumerate}
The homotopy category of this model category is $\Der R$.
\end{thm}

\begin{proof}
Everything follows from Propositions~\ref{prop:models} and~\ref{prop:cotilt-cot-pair} except for the description of cofibrations. We need to prove that $Z$ is cofibrant, i.e.\ $\Ext^1_{\Cpx R}(Z,X) = 0$ for all $X \in \tilde\mcF_0$, \iff $Z \in \Cpx{\mcX_C}$, i.e.\ all components of $Z$ are in $\mcX_C$.

Thanks to~\cite[Lemma 3.1(6)]{Gil04} we know that every cofibrant complex $C$ belongs to $\Cpx{\mcX_C}$.

In order to prove the converse, let $Z \in \Cpx{\mcX_C}$. Since any extension
\[ 0 \la X \la E \la Z \la 0 \quad \textrm{ \st } X \in \tilde\mcF_0 \]
splits in every component, it suffices to prove that $\Hom_{\Htp R}(Z,X) = 0$ for every $X \in \tilde\mcF_0$ (see for instance~\cite[Lemma 1.2]{St13-deconstr}).

To see that, suppose that $X \in \tilde\mcF_0$ and consider the complete cotorsion pair $(\mcP,\mcY)$ in $\Cpx R$ generated by the complexes of the form
\[ \Sigma^{-n}R\dd \quad \cdots \la 0 \la 0 \la R \la 0 \la 0 \la \cdots, \]
where $R$ is in cohomological degree $n$ and $n$ runs over all integers. It is easy to see that the closure of this collection under extensions and coproducts contains a generator, so Proposition~\ref{prop:Quillen-cotorsion} applies. Another application of~\cite[Lemma 1.2]{St13-deconstr} shows that every complex in $\mcY$ is acyclic.

Let now $0 \to Y \to P \to Z \to 0$ be an approximation sequence for $Z$ \wrt $(\mcP,\mcY)$. That is, $P \in \mcP$ and $Y \in \mcY$ is acyclic. Since clearly $\mcP \subseteq \Cpx{\mcX_C}$ and $\Cpx{\mcX_C}$ is resolving in $\Cpx R$, it follows that $Y \in \Cpx{\mcX_C}$. An easy dimension shifting argument shows that if $C$ is big $n$-cotilting, then every module $M \in \ModR$ has $\resdim_{\mcX_C} M \le n$. In particular all cocycles of $Y$ are in $\mcX_C$ and so $Y \in \tilde\mcX_C$. Thus
\[ \Hom_{\Htp R}(Y,X) \cong \Ext^1_{\Cpx R}(\Sigma Y,X) = 0 \]
by \cite[Lemmas 1.2]{St13-deconstr} and~\cite[Theorem 7.2.3]{EJv2} or~\cite[Proposition 7.7]{St13-ICRA}. Similarly $\Hom_{\Htp R}(P,X) = 0$.

Let us now apply the functor $\HOM_R(-,X)\dd \Cpx R \to \Cpx\Ab$ to $0 \to Y \to P \to Z \to 0$. Since all components of $Z$ are in $\mcX_C$ and all components of $X$ are in $\mcF_0$, we deduce that the sequence
\[ 0 \la \HOM_R(Z,X) \la \HOM_R(P,X) \la \HOM_R(Y,X) \la 0 \]
is exact in $\Cpx\Ab$. As both $\HOM_R(P,X)$ and $\HOM_R(Y,X)$ are acyclic complexes, so is $\HOM_R(Z,X)$. In particular $\Hom_{\Htp R}(Z,X) = 0$ as desired.
\end{proof}

We conclude the section by a key corollary, which will be used several times in the sequel.

\begin{cor} \label{cor:compute-in-DerR-cotilt}
Let $C$ be a big cotilting $R$-module, $\mcX_C$ be the cotilting class, and let $Z,X \in \Cpx R$ \st $Z \in \Cpx{\mcX_C}$ and $\Ext^1_{\Cpx R}(Z',X) = 0$ for all $Z' \in \tilde\mcX_C$. The last condition is satisfied in particular when $X$ is a bounded below complex with components in $\Prod C$. Then the canonical morphism
\[ \Hom_\HtpG(Z,X) \la \Hom_\DerG(Z,X) \]
is a natural isomorphism.
\end{cor}

\begin{proof}
We use the description of the model structure in Theorem~\ref{thm:cotilt-models} to deduce that $Z$ is cofibrant and $X$ is fibrant. The rest follows from Corollary~\ref{cor:compute-in-DerG}.
\end{proof}

\section{Cotilting $t$-structures}
\label{sec:t-str}

Given a cotilting module, we need to construct the corresponding tilted Gro\-then\-dieck category. We will do so using basic results from~\cite{BBD82} on $t$-structures and their hearts. If $n = 1$, one could use as in~\cite{CGM07} the $t$-structures coming from torsion pairs in $\ModR$, originally constructed by Happel, Reiten and Smal\o{} in~\cite{HRS96} (see also~\cite[\S5]{BvdB03} or~\cite{SKT11} for a somewhat generalized version). Here we give another construction for an arbitrary $n$ which recovers the $t$-structures from~\cite{CGM07} as a special case if $n=1$.

For the sake of completeness, let us recall the definition of a $t$-structure. It formally resembles that of a complete cotorsion pair and, indeed, there is a direct connection in some cases (see~\cite[\S3]{SaoSt11}).

\begin{defn} \label{defn:t-str}
Let $\mcT$ be a triangulated category. A \emph{$t$-structure} on $\mcT$ is a pair of full subcategories $(\Dtle0,\Dtge0)$ which satisfies the properties (tS1)--(tS3) below. Here we use the customary notation that $\Dtle n = \Sigma^{-n} \Dtle0$ and $\Dtge n = \Sigma^{-n} \Dtge0$ for each $n \in \bbZ$.

\begin{enumerate}
\item[(tS1)] $\Hom_\mcT(A,B) = 0$ for all $A \in \Dtle0$ and $B \in \Dtge1$.
\item[(tS2)] $\Dtle0 \subseteq \Dtle1$ and $\Dtge0 \supseteq \Dtge1$.
\item[(tS3)] For every object $X \in \mcT$ there exist a triangle
\[ \ttle0 X \la X \la \ttge1 X \la \Sigma \ttle0 X. \]
\end{enumerate}
The \emph{heart} of the $t$-structure is the full subcategory $\mcG = \Dtle0 \cap \Dtge0$ of $\mcT$.
\end{defn}

One easily checks that (tS1) and~(tS2) imply that a triangle as in (tS3) is unique up to a unique isomorphism, hence functorial. Standard and general properties of the heart of a $t$-structure, established in~\cite{BBD82}, are summarized in the following proposition.

\begin{prop} \label{prop:heart}
Let $\mcT$ be a triangulated category and $(\Dtle0,\Dtge0)$ be a $t$-structure with the heart $\mcG = \Dtle0 \cap \Dtge0$. Then:
\begin{enumerate}
\item $\mcG$ is an abelian category which is closed under extensions
  in $\mcT$ (that is, given $X,Z \in \mcG$ and a triangle $X \to Y \to Z
  \to \Sigma X$ in $\mcT$, then $Y \in \mcG$).
\item A sequence $0 \to X \overset{f}\to Y \overset{g}\to Z \to 0$ is
  exact in $\mcG$ \iff there is a triangle $X \overset{f}\to Y
  \overset{g}\to Z \to \Sigma X$ in $\mcT$.
\item There is an isomorphism $\Ext^1_\mcG(X,Y) \cong \Hom_\mcT(X,\Sigma Y)$
  which is functorial in both variables.
\end{enumerate}
\end{prop}

\begin{proof}
(1) is a part of \cite[Th\'eor\`eme 1.3.6]{BBD82}, while (2) quickly follows from \cite[Proposition 1.2.2]{BBD82}. Finally, (3) is an immediate consequence of~(1) and~(2).
\end{proof}

Let us now suppose that we have a big cotilting module $C \in \ModR$. We aim to assign to $C$ a $t$-structure $(\Dtle0_C,\Dtge0_C)$ on $\Der R$ \st the heart is a Grothendieck category, and $C$ is contained in the heart and is an injective cogenerator there.

Suppose for a moment that we have $\mcG$ and an injective cogenerator $W \in \mcG$. Then the canonical $t$-structure on $\DerG$ (see~\cite[Exemples 1.3.2(i)]{BBD82}) whose heart is equivalent to $\mcG$ can be described using $W$ as follows:
\begin{align*}
\Dtle0 &= \{ Z \in \DerG \mid \Hom_\DerG(Z,\Sigma^iW) = 0 \textrm{ for all } i<0 \}, \\
\Dtge0 &= \{ X \in \DerG \mid \Hom_\DerG(X,\Sigma^iW) = 0 \textrm{ for all } i>0 \}.
\end{align*}

As we expect an equivalence between $\Der R$ and $\DerG$ which sends $C$ to an injective cogenerator of $\mcG$ (we shall prove that in Section~\ref{sec:heart}), the following definition makes sense:

\begin{nota} \label{nota:cotilt-t-str}
Let $C \in \ModR$ be a big cotilting module. We denote by $\Dtle0_C$ and $\Dtge0_C$ the full subcategories of $\Der R$ given by
\begin{align*}
\Dtle0_C &= \{ Z \in \Der R \mid \Hom_{\Der R}(Z,\Sigma^iC) = 0 \textrm{ for all } i<0 \}, \\
\Dtge0_C &= \{ X \in \Der R \mid \Hom_{\Der R}(X,\Sigma^iC) = 0 \textrm{ for all } i>0 \}.
\end{align*}
\end{nota}

Of course, we must first prove that $(\Dtle0_C,\Dtge0_C)$ indeed is a $t$-structure. We shall do so using the corresponding cotilting model structure and in particular Corollary~\ref{cor:compute-in-DerR-cotilt}. The following is a key lemma giving alternative descriptions of~$\Dtge0_C$.

\begin{lem} \label{lem:right-aisle-cotilt}
Let $C \in \ModR$ be a big cotilting module, $\mcX_C$ the corresponding cotilting class (Definition~\ref{defn:cotilt-class}), and let $X \in \Der R$ be a complex. Then the following are equivalent:
\begin{enumerate}
\item $X \in \Dtge0_C$ as in Notation~\ref{nota:cotilt-t-str}.
\item $X$ is isomorphic in $\Der R$ to a complex of the form
\[ \cdots \la 0 \la 0 \la X^0 \la X^1 \la X^2 \la \cdots \]
with $X^i \in \mcX_C$ for all $i \ge 0$.
\item $X$ is isomorphic in $\Der R$ to a complex as in \emph{(2)}, but with $X^i \in \Prod C$ for all $i \ge 0$.
\end{enumerate}
\end{lem}

\begin{proof}
(3) $\implies$ (2) is obvious and (2) $\implies$ (1) is an immediate consequence of Corollary~\ref{cor:compute-in-DerR-cotilt}.

We shall focus on (1) $\implies$ (3). To this end, let $X \in \Dtge0_C$ and, in view of Theorem~\ref{thm:cotilt-models}, we shall without loss of generality assume that $X \in \Cpx{\mcX_C}$. Similarly to~\cite[Lemma I.4.6(i)]{Hart66}, we will inductively construct a map $f\dd X \to D$ of complexes of $R$-modules
\[
\begin{CD}
\cdots @>>>  X^{-2} @>{\dif^{-2}}>> X^{-1} @>{\dif^{-1}}>> X^0 @>{\dif^0}>> X^1 @>{\dif^{1}}>>  X^2  @>>> \cdots  \\
@.           @VVV                    @VVV               @V{f^0}VV        @V{f^1}VV           @V{f^2}VV            \\
\cdots @>>>    0    @>>>               0   @>>>            D^0 @>>>         D^1 @>>>            D^2  @>>> \cdots  \\
\end{CD}
\]
\st $D^i \in \Prod C$ for all $i \ge 0$. We take for $f^0$ the composition of $X^0 \to \Coker \dif^{-1}$ with a $\Prod C$-preenvelope $e^0\dd \Coker \dif^{-1} \to D^0$ (recall Definition~\ref{defn:approx}). Such a preenvelope always exists---we can for instance take the obvious morphism $\Coker \dif^{-1} \to C^{\Hom_R(\Coker \dif^{-1}, C)}$. Having constructed $f^i$ for $i \ge 0$, we denote by $c^i\dd D^i \to C^i$ the cokernel of the map $D^{i-1} \to D^i$ and consider the pushout
\begin{equation} \label{eq:pushout}
\begin{CD}
   X^i @>{\dif^i}>> X^{i+1}     \\
@V{c^if^i}VV        @VV{g^{i+1}}V  \\
   C^i @>{p^i}>>    P^{i+1}.     \\
\end{CD}
\end{equation}
We set $f^{i+1}$ equal to the composition of $g^{i+1}$ with a $\Prod C$-preenvelope $e^{i+1}\dd P^{i+1} \to D^{i+1}$, and the differential $D^i \to D^{i+1}$ is given by the obvious composition
\[ D^i \mapr{c^i} C^i \mapr{p^i} P^{i+1} \mapr{e^{i+1}} D^{i+1}. \]
All the maps used in the inductive step are depicted in the following diagram:
\[
\xymatrix{
X^{i-1} \ar[r]^{\dif^{i-1}} \ar[d]_{f^{i-1}} & X^i \ar[r]^{\dif^i} \ar[d]_{f^i} & X^{i+1} \ar[dd]^(.66){g^{i+1}} \ar[ddr]^{f^{i+1}} \\
D^{i-1} \ar[r] & D^i \ar[d]_{c^i} & \\
& C^i \ar[r]^{p^i} & P^{i+1} \ar[r]^{e^{i+1}} & D^{i+1}.
}
\]

We claim that $\Hom_{\Htp R}(f,\Sigma^jC)$ is an isomorphism for all $j \in \bbZ$. For $j>0$ this is clear from the assumption on $X \in \Dtge0_C$. If $j\le 0$, suppose that we have a morphism $h\colon X \to \Sigma^jC$ in $\Cpx R$. Our task is to prove that
\begin{enumerate}
\item[(a)] $h$ factors through $f\colon X \to D$, and
\item[(b)] if $h$ is null-homotopic, then so is any such factorization $t\colon D \to \Sigma^jC$.
\end{enumerate}

Regarding (a), note that $h$ is given by its $(-j)$-th component $h^{-j}\dd X^{-j} \to C$ and that $h^{-j} \dif^{-j-1} = 0$. Using the pushout property of~\eqref{eq:pushout} for $i = -j-1$, we construct a map $\ell^{-j}\dd P^{-j} \to C$ such that $\ell^{-j} g^{-j} = h^{-j}$ and $\ell^{-j} p^{-j-1} = 0$. We can factor $\ell^{-j}$ further through the $\Prod C$-preenvelope $e^{-j}\dd P^{-j} \to D^{-j}$ to obtain a morphism of $R$-modules $t^{-j}\dd D^{-j} \to C$. It follows that the composition of $t^{-j}$ with the differential $D^{-j-1} \to D^{-j}$ vanishes, so that $t^{-j}$ is the $(-j)$-th component of a homomorphism of complexes $t\dd D \to \Sigma^jC$ satisfying $h = tf$. In particular, $t$ is a factorization of $h$ through $f$.

To prove (b), let us assume that $h$ is nullhomotopic, where we denote the only relevant component of the null-homotopy by $s^{-j+1}\dd X^{-j+1} \to C$. Given a fixed $t\dd D \to \Sigma^jC$ with $h=tf$, we will now show that $t$ is also null-homotopic. Note that in this situation the $(-j)$-th component $t^{-j}\colon D^{-j} \to C$ satisfies $t^{-j}f^{-j} = h^{-j} = s^{-j+1}\dif^{-j}$. Since $t^{-j}$ is a component of a morphism of complexes, its composition with the differential $D^{-j-1} \to D^{-j}$ vanishes and so there is a map $u^{-j}\dd C^{-j} \to C$ \st $t^{-j} = u^{-j} c^{-j}$. Now we use the pushout from~\eqref{eq:pushout} for $i = -j$ to construct a map $v^{-j+1}\dd P^{-j+1} \to C$ \st $v^{-j+1}p^{-j} = u^{-j}$ and $v^{-j+1} g^{-j+1} = s^{-j+1}$. We further factor $v^{-j+1}$ throught the preenvelope $e^{-j+1}\dd P^{-j+1} \to D^{-j+1}$, obtaining a morphism of $R$-modules $\tilde s^{-j+1}\dd D^{-j+1} \to C$. To summarize, we have constructed $\tilde s^{-j+1}\dd D^{-j+1} \to C$ whose composition with the differential $D^{-j} \to D^{-j+1}$ equals $t^{-j}$. This is none other than a null-homotopy of $t\dd D \to \Sigma^jC$, finishing the proof of the claim.

If we now denote by $Z \in \Cpx R$ the mapping cone of $f\dd X \to D$, we have $\Hom_{\Htp R}(Z,\Sigma^iC) = 0$ for all $i \in \bbZ$ and, owing to Corollary~\ref{cor:compute-in-DerR-cotilt}, also
\[ \Hom_{\Der R}(Z, \Sigma^iC) = 0 \quad \textrm{ for all } i \in \bbZ. \]
Let $Q$ be an injective cogenerator for $\ModR$ as in (BC3) of Definition~\ref{defn:big-cotilt}. It follows that $\Hom_{\Der R}(Z, \Sigma^iQ) = 0$ for all $i \in \bbZ$, showing that $Z \cong 0$ in $\Der R$. Thus, $f$ becomes an isomorphism in $\Der R$ and $X$ satisfies (3).
\end{proof}

Now we can prove that we have indeed assigned a $t$-structure to a big cotilting module.

\begin{thm} \label{thm:cotilt-t-structure}
Let $R$ be a ring and $C$ be a big cotilting $R$-module. Then $(\Dtle0_C,\Dtge0_C)$ from Notation~\ref{nota:cotilt-t-str} is a $t$-structure on the derived category $\Der R$.
\end{thm}

\begin{rem} \label{rem:t-str-comm-noe}
For $R$ commutative noetherian, a method to assign a $t$-structure to a big cotilting module have been recently studied in~\cite{ASa13}; see for instance Corollary~3.5 there. This assignment was obtained indirectly, using two classification results: Compactly generated $t$-structures on $\Der R$ have been classified in~\cite{AJSa10} and cotilting classes corresponding to big cotilting $R$-modules have been classified in~\cite{APST12}. It turns out that the set parameterizing the cotilting classes is a subset of the one parameterizing the $t$-structures. In view of~\cite[Proposition 3.6]{ASa13}, our theorem generalizes this assignment to rings which are not necessarily commutative or noetherian.
\end{rem}

\begin{proof}[Proof of Theorem~\ref{thm:cotilt-t-structure}]
Property (tS2) from Definition~\ref{defn:t-str} is clear from the construction of $\Dtle0_C$ and~$\Dtge0_C$.

In order to prove (tS1), suppose that $Z \in \Dtle0_C$ and $X \in \Dtge1_C$. In view of the previous lemma, we can assume that $X$ is of the form
\[ \cdots \la 0 \la 0 \la X^1 \la X^2 \la \cdots \la X^n \la X^{n+1} \la X^{n+2} \la \cdots \]
with $X^i \in \Prod C$ for all $i > 0$. For every $n \ge 1$, we shall denote by $X_n$ the brutally truncated complex
\[ \cdots \la 0 \la 0 \la X^1 \la X^2 \la \cdots \la X^n \la 0 \la 0 \la \cdots \]

As $\Hom_{\Der R}(Z,\Sigma^iC) = 0$ for all $i<0$, it is straightforward to show by induction on $n$ that
\[ \Hom_{\Der R}(Z,X_n) = 0 = \Hom_{\Der R}(Z,\Sigma\inv X_n) \quad \textrm{ for all } n \ge 1. \]
Now we use~\cite[Remark 2.3]{BN93} to show that there is a triangle in $\Der R$ of the form
\[ \prod_{n \ge 0} \Sigma\inv X_n \la X \la \prod_{n \ge 0} X_n \la \prod_{n \ge 0} X_n. \]
Hence $\Hom_{\Der R}(Z,X) = 0$ and (tS1) holds.

In order to prove (tS3), we start with an arbitrary complex $X \in \Der R$. Using Theorem~\ref{thm:cotilt-models}, we can assume that $X \in \Cpx{\mcX_C}$, where $\mcX_C$ is the cotilting class corresponding to $C$. Now we construct a map $f\dd X \to D$ in $\Cpx R$
\[
\begin{CD}
\cdots @>>>  X^{-2} @>{\dif^{-2}}>> X^{-1} @>{\dif^{-1}}>> X^0 @>{\dif^0}>> X^1 @>{\dif^{-1}}>> X^2  @>>> \cdots  \\
@.           @VVV                    @VVV               @V{f^0}VV        @V{f^1}VV           @V{f^2}VV            \\
\cdots @>>>    0    @>>>               0   @>>>            D^0 @>>>         D^1 @>>>            D^2  @>>> \cdots  \\
\end{CD}
\]
using exactly the same method as in the proof of Lemma~\ref{lem:right-aisle-cotilt}. Since $X$ not necessarily in $\Dtle0_C$, this time we only get that
\[ \Hom_{\Htp R}(f, \Sigma^iC) \quad \textrm{ is an isomorphism for all } i \le 0. \]
In view of Corollary~\ref{cor:compute-in-DerR-cotilt}, the same is true for $\Hom_{\Der R}(f, \Sigma^iC)$. Let $F$ be the fiber of $f$ in $\Der R$. That is, we have a triangle
\[ \Sigma\inv D \la F \la X \mapr{f} D. \]
Then necessarily $\Hom_{\Der R}(F,\Sigma^iC) = 0$ for all $i < 0$. We claim that the group $\Hom_{\Der R}(F,C)$ vanishes as well. Indeed, if we apply $\Hom_{\Der R}(-,C)$ to the triangle above, we get an exact sequence of abelian groups
\[ \Hom_{\Der R}(D,C) \mapr{f'} \Hom_{\Der R}(X,C) \la \Hom_{\Der R}(F,C) \la \Hom_{\Der R}(\Sigma\inv D,C) \]
Now $f'$ is an isomorphism and $\Hom_{\Der R}(\Sigma\inv D,C) = 0$ since $\Sigma\inv D \in \Dtge1_C$ by Lemma~\ref{lem:right-aisle-cotilt}. This proves the claim and in particular we have $F \in \Dtle{-1}_C$ and $D \in \Dtge0_C$. Hence, by constructing a triangle as above for $\Sigma X$ rather than for $X$ and by rotating it correspondingly we get a triangle as in (tS3).
\end{proof}

Now we can give a central definition for the rest of the paper. We will use it to prove a converse of Theorem~\ref{thm:cotilt-from-tilt}.

\begin{defn} \label{defn:tilted-catg}
Let $R$ be a ring and $C \in \ModR$ be a big cotilting module. Then the \emph{tilted abelian category} $\mcG$ corresponding to $C$ is defined to be the heart of the cotilting $t$-structure described in Notation~\ref{nota:cotilt-t-str}. That is,
\[ \mcG = \Dtle0_C \cap \Dtge0_C = \{ X \in \Der R \mid \Hom_{\Der R}(X,\Sigma^iC) = 0 \textrm{ for all } i \ne 0 \}. \]
\end{defn}

\section{Equivalence of derivators}
\label{sec:derivators}

In this section we generalize one of the main results from~\cite{CGM07} and show that, given a big cotilting module $C$, the tilted category $\mcG$ (Definition~\ref{defn:tilted-catg}) is derived equivalent to $\ModR$. In fact, we shall prove more: the diagram categories $(\ModR)^I$ and $\mcG^I$, where $I \in \Cat$ is an arbitrary small category, are also derived equivalent and these equivalences are compatible with the restriction functors induced by morphisms in $\Cat$. Formally this says that the canonical derivators (in the sense of~\cite{CisNee08,Gr12,GrothDerivators,Malt07-derivators}, to be defined) of $R$ and $\mcG$ are equivalent.

\subsection{The $2$-category of prederivators}
\label{subsec:2-cat-prederiv}

The basic definition is formally easy.

\begin{defn} \label{defn:prederiv}
A \emph{prederivator} is a strict $2$-functor $\bbD\dd \Cat\op \to \CAT$. Here, $\Cat$ stands for the $2$-category of all small categories and $\CAT$ for the ``$2$-category'' of all not necessarily small categories. Our convention for $\Cat\op$, following~\cite{Gr12}, is that the direction of $1$-morphisms (functors) is formally inverted and the direction of $2$-morphisms (transformations) is kept. The adjective ``strict'' refers to the fact that $\bbD$ is required to preserve the compositions of $1$-morphisms strictly rather than only up to natural equivalence.
\end{defn}

A derivator will be defined as a prederivator satisfying certain axioms. Before introducing these, we shall discuss how (pre)derivators occur in our setup.

\begin{constr} \label{constr:der-abel}
Let $\mcA$ be an abelian category. We shall construct a prederivator $\bbD_\mcA\dd \Cat\op \to \CAT$ as follows. Given $I \in \Cat$, we shall put $\bbD_\mcA(I) = \Der{\mcA^I}$, the unbounded derived category of the abelian diagram category $\mcA^I$. Note that $\Der{\mcA^I}$ is typically not equivalent to $\Der\mcA^I$, although there always exists a canonical functor
\[ d_I\dd \Der{\mcA^I} \la \Der\mcA^I. \]
We call such a functor a \emph{diagram functor} and refer to~\cite[\S2]{Malt07-derivators} or~\cite[\S1.10]{CisNee08} for a general formal construction.

Given a $1$-morphism $u\dd I \to J$ in $\Cat$, the scalar restriction functor $\mcA^J \to \mcA^I$ is clearly exact, so it induces a functor
\[ u^*\dd \bbD_\mcA(J) \la \bbD_\mcA(I) \]
We shall take this $u^*$ as the value of $\bbD_\mcA(u)$.

Finally, given $u,v\dd I \to J$ in $\Cat$ and a natural transformation $\alpha\dd u \Rightarrow v$, there is an obvious natural transformation
\[ \alpha^*\dd u^* \Longrightarrow v^*, \]
which we take for the value of $\bbD_\mcA(\alpha)$.

If $\mcA = \ModR$ is a module category, we shall simply write $\bbD_R$ instead of $\bbD_{\ModR}$.
\end{constr}

\begin{rem}
In principle, we may run out of our set theoretic universe since $\bbD_\mcA(I)$ may not be locally small, i.e.\ the collections of morphism between some pairs of objects of $\bbD_\mcA(I)$ need not be small. Although such pathologies do exist (cf.~\cite[Exercise 1, p. 131]{Fr64}), the problem does not occur if $\mcA$ is a Grothendieck category.
\end{rem}

The class of prederivators themselves forms a $2$-category. Saying that, we disregard for the moment the set-theoretic problem that the class of all $1$-morphisms between given prederivators $\bbD$ and $\bbD'$ may not be small or, worse, may not be a legal class in our universe. However, this will not cause any problems as we will always consider only particular $1$- or $2$-morphisms rather than the class of all of them.

\begin{defn} \label{defn:2-cat-prederiv}
A \emph{$1$-morphism} (or simply a \emph{morphism}) $F\dd \bbD \to \bbD'$ of prederivators is given by the following data:
\begin{enumerate}
\item A collection of functors $F_I\dd \bbD(I) \to \bbD'(I)$, one for each $I \in \Cat$.
\item A collection of natural equivalences $\gamma_u\dd u^* \circ F_J \Rightarrow F_I \circ u^*$ indexed by $1$-morphisms $u\dd I \to J$ in $\Cat$. That is, we have squares of functors
\[
\begin{CD}
 \bbD(J) @>{F_J}>> \bbD'(J)     \\
@V{u^*}VV          @VV{u^*}V    \\
 \bbD(I) @>{F_I}>> \bbD'(I)     \\
\end{CD}
\]
commutative up to the equivalence $\gamma_u$.
\end{enumerate}

This datum is required to satisfy the following three coherence conditions. Given a pair of composable functors $I \overset{u}\to J \overset{v}\to K$ in $\Cat$ and a natural transformation $\alpha\dd u_1 \Rightarrow u_2$ between $u_1,u_2\dd I \to J$ in $\Cat$, we have the equalities
\begin{enumerate}
\item[(a)] $\gamma_{\id_J} = \id_{F_J}$ as transformations $F_J \Rightarrow F_J$,
\item[(b)] $\gamma_{v \circ u} = \gamma_uv^* \circ u^*\gamma_v$ as transformations $u^* \circ v^* \circ F_K \Rightarrow F_I \circ u^* \circ v^*$, and
\item[(c)] $F_I\alpha^* \circ \gamma_{u_1} = \gamma_{u_2} \circ \alpha^*F_J$ as transformations $u_1^* \circ F_J \Rightarrow F_I \circ u_2^*$.
\end{enumerate}

If $F,G\dd \bbD \to \bbD'$ are two morphisms of prederivators, a \emph{$2$-morphism} (or a \emph{natural transformation}) $\tau\dd F \Rightarrow G$ is collection of natural transformations $\tau_I\dd F_I \to G_I$, where $I$ runs over small categories, satisfying the following compatibility condition. Given a functor $u\dd I \to J$ in $\Cat$, then $\tau_Iu^* \circ \gamma_u^F = \gamma_u^G \circ u^*\tau_J$ as transformations $u^* \circ F_J \Rightarrow G_I \circ u^*$, where the superscripts at $\gamma_u$ distinguish the transformation which is a part of the defining datum of $F$ from the one belonging to $G$.
\end{defn}

\begin{rem}
The definitions here coincide with the definitions of pseudo-natural transformations of $2$-functors and modifications of pseudo-natural transformations; see~\cite[Definitions 7.5.2 and 7.5.3]{Bor94-vol1}.
\end{rem}

Although the definition above may seem complicated, there is a natural class of easy examples relevant for our case.

\begin{expl} \label{expl:exact-functor}
Let $F\dd \mcA \to \mcB$ be an exact functor between abelian categories. Then $F$ induces an exact functor $\mcA^I \to \mcB^I$ for every $I \in \Cat$ and, in turn, a functor $F_I\dd \bbD_\mcA(I) \to \bbD_\mcB(I)$. One easily checks that $u^* \circ F_J = F_I \circ u^*$ for every $1$-morphism $u\dd I \to J$, so the collection $(F_I \mid I \in \Cat)$ together with the identity transformations forms a morphism $\bbD_\mcA \to \bbD_\mcB$ of the corresponding prederivators. Abusing the notation a little, we shall denote this morphism also by $F$.
\end{expl}

As we are interested in equivalences of prederivators, we shall also define precisely what we mean by these.

\begin{defn} \label{defn:equiv-prederiv}
A functor of prederivators $F\dd \bbD \to \bbD'$ is called an \emph{equivalence} if $F_I\dd \bbD(I) \to \bbD(I')$ is a usual equivalence of categories for every $I \in \Cat$.
\end{defn}

\begin{rem} \label{rem:equiv-prederiv}
As prederivators form a $2$-category, one can also define an equivalence internally to the $2$-category of prederivators; see~\cite[Definition 7.1.2]{Bor94-vol1}. It can be checked that this leads to a notion equivalent to Definition~\ref{defn:equiv-prederiv}.
\end{rem}

\subsection{Derivators}
\label{subsec:derivators-def}

Now we turn our attention to derivators. Let us introduce some notations and terminology first. We denote by $\unit$ the terminal category with one object and one morphism and, given $n \in \bbN$, we denote by $[n]$ the totally ordered set $\{0,1,\dots,n\}$ viewed as a small category (so that $[0] \cong \unit$).

Given $I \in \Cat$ and an object $i \in I$, there is a unique functor $\unit \to I$ sending the only object of $\unit$ to $i$. We shall denote this functor also by $i$. Given a prederivator $\bbD$ and a morphism $f\dd X \to Y$ in $\bbD(I)$, we shall write $f_i\dd X_i \to Y_i$ for the image of $f$ under $i^*\dd \bbD(I) \to \bbD(\unit)$.

If $I \in \Cat$, then the objects of $\bbD(I)$ are called \emph{coherent diagrams} of shape $I$, while the usual diagrams $\bbD(\unit)^I$ are called \emph{incoherent diagrams}. As noted in Construction~\ref{constr:der-abel}, there is always a canonical \emph{diagram functor} $d_I\dd \bbD(I) \to \bbD(\unit)^I$ which is, however, usually far from being an equivalence. More generally, if $I,J \in \Cat$, we have a \emph{partial diagram functor} $d_{I,J}\dd \bbD(I\times J) \to \bbD(I)^J$; see~\cite[\S1.10]{CisNee08}.

\begin{defn}[\cite{CisNee08,Gr12,GrothDerivators,Malt07-derivators}] \label{defn:derivator}
A prederivator $\bbD\dd \Cat\op \to \CAT$ is a \emph{derivator} if it satisfies the axioms (Der1)--(Der4) below.
\begin{enumerate}
\item[(Der1)] Given a set $S$ and a collection $(I_s \mid s \in S)$ of small categories, the canonical functor
\[ \bbD\Big(\coprod_{s\in S} I_S\Big) \la \prod_{s \in S} \bbD(X_s) \]
is an equivalence of categories.

\item[(Der2)] Let $I \in \Cat$ and $f\dd X \to Y$ be a morphism in $\bbD(I)$. Then $f$ is an isomorphism \iff $f_i\dd X_i \to Y_i$ is an isomorphism for each $i \in I$.
\item[(Der3)] Given a functor $u\dd I \to J$ in $\Cat$, then $u^*\dd \bbD(J) \to \bbD(I)$ has both a left adjoint $u_!\dd \bbD(I) \to \bbD(J)$ and a right adjoint $u_*\dd \bbD(I) \to \bbD(J)$
\end{enumerate}
If $J = \unit$ and $p\dd I \to \unit$ is the unique functor, it is common to use special notation and terminology. The functor $p_!\dd \bbD(I) \to \bbD(\unit)$ is denoted by $\hocolim_I$ and called the \emph{homotopy colimit} functor, while $p_*$ is denoted by $\holim_I$ and is called the \emph{homotopy limit} functor.
\begin{enumerate}
\item[(Der4)] Given a functor $u\dd I \to J$ in $\Cat$ and objects $X \in \bbD(I)$ and $i \in I$, there are base change formulas which allow to canonically identify $u_!(X)_i$ and $u_*(X)_i$ with homotopy colimits and limits of certain comma categories, respectively. Rather than giving the details here, we refer to~\cite[\S2]{Malt07-derivators}, \cite[Definition 1.11]{CisNee08} or~\cite[Definition 1.5]{Gr12}.
\end{enumerate}

The derivator $\bbD$ is called \emph{strong} if it in addition satisfies
\begin{enumerate}
\item[(Der5)] The partial diagram functors $d_{[1],I}\dd \bbD([1] \times I) \to \bbD(I)^{[1]}$ are full and essentially surjective for all $I \in \Cat$.
\end{enumerate}
\end{defn}

Note that, as a consequence of (Der1), each $\bbD(I)$ has small coproducts and products and, in particular, an initial and a terminal object. When we will speak of morphisms and transformations of (strong) derivators, we shall consider the corresponding full sub-$2$-category of the $2$-category of prederivators (Definition~\ref{defn:2-cat-prederiv}). Note also that a prederivator $\bbD$ which is equivalent to a (strong) derivator in the sense of Definition~\ref{defn:equiv-prederiv} is itself a (strong) derivator.

There are abundance of examples of derivators, owing to the following result.

\begin{prop} \label{prop:models-to-derivators}
Let $\mcA$ together with $(\cof,\we,\fib)$ be a model category. Then there is a strong derivator $\bbD_\mcA$ given as follows: $\bbD(I) = \mcA^I\big[(\we^I)\inv\big]$ for $I \in \Cat$, where $\we^I$ consists of the morphism in $\mcA^I$ which are component wise weak equivalences. If $u\dd I \to J$ is a functor in $\Cat$, then $u^*\dd \bbD_\mcA(J) \to \bbD_\mcA(I)$ is given by the precomposition with $u$ and the action of $\bbD_\mcA$ on natural transformations is given in the obvious way.

In particular, if $\mcG$ is a Grothendieck category, then the prederivator $\bbD_\mcG$ from Construction~\ref{constr:der-abel} is a strong derivator.
\end{prop}

\begin{proof}
The first part is due to Cisinski~\cite{Cis03}, while a simpler proof for combinatorial model categories (in the sense of~\cite[\S2]{Dug01}) can be found in~\cite[Proposition 1.30]{Gr12}. The second part follows from Example~\ref{expl:inj-model} which shows that we have an (even combinatorial) model structure on $\Cpx\mcG$.
\end{proof}

However, the derivator $\bbD_\mcG$ for a Grothendieck category $\mcG$ satisfies more properties. We have not reflected yet that all the categories $\bbD_\mcG(I)$ are triangulated (and in particular additive). This leads to the notion of a stable derivator.

\begin{defn}[\cite{Gr12}] \label{defn:stable-deriv}
A derivator $\bbD$ is \emph{pointed} if:
\begin{enumerate}
\item[(Der6)] The category $\bbD(\unit)$ is a pointed category (i.e.\ the initial and terminal objects are isomorphic).
\end{enumerate}

A pointed derivator is called \emph{stable} if:
\begin{enumerate}
\item[(Der7)] The full subcategory of $\bbD([1] \times [1])$ consisting of cartesian squares in the sense of~\cite[Definition 3.9]{Gr12} coincides with the full subcategory of cocartesian squares.
\end{enumerate}
\end{defn}

The following result is crucial for us.

\newpage
\begin{prop} \label{prop:stable-deriv-Grothendieck} ~
\begin{enumerate}
\item Let $\mcG$ be a Grothendieck category and $\bbD_\mcG$ the corresponding strong derivator from Construction~\ref{constr:der-abel}. Then $\bbD_\mcG$ is stable.

\item If $\bbD$ is a strong stable derivator, then $\bbD(I)$ is canonically a triangulated category for each $I \in \Cat$ and $u_!,u^*,u_*$ are triangulated functors for each functor $u\dd I \to J$ in $\Cat$.

\item If $F\dd \bbD \to \bbD'$ is an equivalence of prederivators and one of $\bbD,\bbD'$ is a strong stable derivator, then so is the other and all the component functors $F_I$, $I \in \Cat$, are triangle equivalences.
\end{enumerate}
\end{prop}

\begin{proof}
Parts (1) and (2) were announced by Maltsiniotis~\cite{Malt07-derivators} and a full proof was given by Groth in~\cite[Theorem 4.16 and Corollary 4.19]{Gr12}. Part (3) follows from \cite[Proposition 4.18]{Gr12}.
\end{proof}

\subsection{Blueprints for derivators}
\label{subsec:blueprints}

Given a Grothendieck category $\mcG$, it turns out that the stable derivator $\bbD_\mcG$ from Construction~\ref{constr:der-abel} is often fully determined by a suitable extension closed full subcategory of $\mcG$. This will allow us to compare $\bbD_R$ and $\bbD_\mcG$, where $R$ is a ring and $\mcG$ is the tilted abelian category with respect to a big cotilting $R$-module.

In order to formalize this observation, we will briefly recall exact categories. This concept is originally due to Quillen, but the common reference is~\cite[Appendix A]{Kel90} and an extensive treatment is given in~\cite{Bu10}. An \emph{exact category} is an additive category $\mcE$ together with a distinguished class of diagrams of the form
\[ 0 \la X \mapr{i} Y \mapr{d} Z \la 0, \]
called \emph{conflations}, satisfying certain axioms which make conflations behave similar to short exact sequences in an abelian category and allow to define Yoneda Ext groups with usual properties. Adopting the terminology from~\cite{Kel90}, the map in a conflation denoted by $i$ above is called \emph{inflation} while $d$ is referred to as \emph{deflation}. Morally, an exact category is an extension closed subcategory of an abelian category, which is made precise in the following statement.

\begin{prop}[\cite{Kel90,Bu10}] \label{prop:exact-categories} ~
\begin{enumerate}
\item Let $\mcA$ be an abelian category and $\mcE \subseteq \mcA$ be an extension closed subcategory. Then $\mcE$, considered together with all short exact sequences in $\mcA$ whose all terms belong to $\mcE$, is an exact category.
\item Every small exact category arises up to equivalence as an extension closed subcategory of an abelian category in the sense of~(1).
\end{enumerate}
\end{prop} 

In particular, given an exact category we can consider its derived category $\Der\mcE$; see~\cite[\S10]{Bu10}. Moreover, given $I \in \Cat$, the diagram category $\mcE^I$ is naturally an exact category with conflations defined component wise. Thus, ignoring possible set theoretic issues, attached to an exact category $\mcE$ we have a prederivator $\bbD_\mcE$ defined analogously to the prederivator of an abelian category in Construction~\ref{constr:der-abel}. That is, $\bbD_\mcE(I) = \Der{\mcE^I}$.

Further, given an inclusion $\mcE \subseteq \mcA$ as in Proposition~\ref{prop:exact-categories}(1), we obtain similarly as in Example~\ref{expl:exact-functor} an induced morphism of prederivators $F\dd \bbD_\mcE \to \bbD_\mcA$. Generalizing the results from~\cite{BvdB03,SKT11}, we establish a criterion for $F$ being an equivalence.

\begin{prop} \label{prop:blueprint}
Let $\mcA$ be an abelian category and $\mcE \subseteq \mcA$ be a full subcategory \st
\begin{enumerate}
\item $\mcE$ is resolving (in particular $\mcE$ is extension closed, see Definition~\ref{defn:resolving}) and
\item there is an integer $n\ge 0$ \st each $A \in \mcA$ has $\resdim_\mcE A \le n$.
\end{enumerate}
Then, viewing $\mcE$ as an exact category as in Proposition~\ref{prop:exact-categories}(1), the induced functor $\Der\mcE \to \Der\mcA$ is a triangle equivalence.

If, moreover, $\mcE$ is functorially resolving in $\mcA$, then the induced morphism of prederivators $F\dd \bbD_\mcE \to \bbD_\mcA$ is an equivalence.
\end{prop}

\begin{rem} \label{rem:blueprint-dual}
It will follow from the proof that the same conclusion holds under formally dual assumptions on $\mcE$, that is if $\mcE$ is (functorially) coresolving in $\mcA$ and the $\mcE$-coresolution dimension of objects of $\mcA$ is uniformly bounded.
\end{rem}

\begin{proof}
The first part follows by the same argument as~\cite[Lemma 5.4.2]{BvdB03}. Namely, we use Proposition~\ref{prop:res-dim}(1) to observe that $\mcE$ satisfies the assumptions for the dual version of~\cite[Lemma I.4.6(2)]{Hart66}. Thus, every complex $X \in \Cpx\mcA$ admits a surjective quasi-isomorphism $f\dd Y \to X$ with $Y \in \Cpx\mcE$. Moreover, the bound on $\mcE$-resolution dimension implies that a complex $X \in \Cpx\mcE$ is acyclic in the sense of~\cite[\S10]{Bu10} \iff it is acyclic as a complex in $\Cpx\mcA$. Thus, $\Der\mcE \to \Der\mcA$ is a triangle equivalence exactly as in~\cite[Lemma 5.4.2]{BvdB03}. A little more detailed argument for the latter step is provided in the proof of~\cite[Lemma 3.8]{SKT11}.

For the second part, note that if $\mcE$ is functorially generating in $\mcA$, then obviously $\mcE^I$ is generating in $\mcA^I$ for every $I \in \Cat$. In particular, $\mcE^I$ satisfies conditions (1) and (2) from the statement as a full subcategory of $\mcA^I$, and so the induced functor $\Der{\mcE^I} \to \Der{\mcA^I}$ is an equivalence. Hence the morphism of prederivators $F\dd \bbD_\mcE \to \bbD_\mcA$ is an equivalence by Definition~\ref{defn:equiv-prederiv}.
\end{proof}

This motivates the following definition.

\begin{defn} \label{defn:blueprint}
Let $\mcA$ be an abelian category and let $\bbD_\mcA$ be the prederivator as in Construction~\ref{constr:der-abel}. A full subcategory $\mcE \subseteq \mcA$ which is functorially resolving and \st the $\mcE$-resolution dimension of objects of $\mcA$ is bounded by some $n \in \bbN$, is called a \emph{blueprint} for $\bbD_\mcA$. If $\mcE$ is functorially coresolving and the $\mcE$-coresolution dimension is uniformly bounded, $\mcE$ is called the \emph{dual blueprint} for $\bbD_\mcA$.
\end{defn}

\subsection{The cotilting equivalence of derivators}
\label{subsec:cotilt-equiv}

Finally we will prove the promised equivalence of derivators induced by a big cotilting module. In retrospection the equivalence will follow from the Quillen equivalence in \S\ref{subsec:QE}; we refer to~\cite[Example 2.10]{Gr12} and also to~\cite{Ren09} for a conceptual explanation. However, we need a triangle equivalence between $\Der R$ and $\Der\mcG$ to prove that $\mcG$ is a Grothendieck category with a classical tilting object in the first place, and the equivalence of the corresponding derivators comes almost for free.

Let $R$ be a ring, $C \in \ModR$ a big cotilting module and $\mcG$ the tilted abelian category in the sense of Definition~\ref{defn:tilted-catg}. One easily observes that $\mcX_C$, the associated cotilting class (Definition~\ref{defn:cotilt-class}), is an extension closed full subcategory of both $\ModR$ and $\mcG$ and that $\mcG \cap \ModR = \mcX_C$.

We shall prove that $\mcX_C$ is a blueprint for $\bbD_R$ and a dual blueprint for $\bbD_\mcG$. The first part is rather well known.

\begin{prop} \label{prop:X_C-blueprint}
Let $R$, $C$ and $\mcX_C \subseteq \ModR$ be as above. Then $\mcX_C$ is a blueprint for $\bbD_R$. In particular, $\bbD_{\mcX_C}$ is a strong stable derivator equivalent to $\bbD_R$.
\end{prop}

\begin{proof}
The class $\mcX_C$ is functorially resolving by Proposition~\ref{prop:cotilt-cot-pair}. The bound on resolution dimension of $R$-modules \wrt $\mcX_C$ follows by dimension shifting from the fact that the injective dimension of modules in the class $\mcF_0$ of the complete cotorsion pair $(\mcX_C,\mcF_0)$ is bounded by the injective dimension of $C$; see~\cite[Lemma 8.1.4(a)]{GT06}. This argument can be traced back to~\cite[Theorem 5.4]{AR91}.

Hence $\mcX_C$ is a blueprint for $\bbD_R$ and the last statement follows from Propositions~\ref{prop:blueprint} and~\ref{prop:stable-deriv-Grothendieck}.
\end{proof}

The proof that $\mcX_C$ is a dual blueprint for $\bbD_\mcG$ is more laborious.

\begin{lem} \label{lem:X_C-coresolving}
Let $R$, $C$, $\mcX_C$ and $\mcG$ be as above. Then both $\mcX_C$ and $\Prod C$ are coresolving in $\mcG$.
\end{lem}

\begin{proof}
Both $\mcX_C$ and $\Prod C$ are clearly closed under retracts and extensions in $\mcG$.

Next we claim that $\mcX_C$ is closed under cokernels of monomorphisms in $\mcG$. To this end, let $0 \to X \to Y \to Z \to 0$ be a short exact sequence in $\mcG$ with $X,Y \in \mcX_C$, so in particular $X,Y \in \ModR$. Thanks to Proposition~\ref{prop:heart}(2), there exists a triangle in $\Der R$ of the form
\[ X \mapr{f} Y \la Z \la \Sigma X \]
and the complex $Z$ is up to isomorphism in $\Der R$ of the form
\[ Z\dd \quad \cdots \la 0 \la 0 \la X \mapr{f} Y \la 0 \la 0 \la \cdots, \]
with $X$ in cohomological degree $-1$ and $Y$ in degree $0$. This complex $Z$ is also cofibrant \wrt the cotilting model structure on $\Cpx R$; see Theorem~\ref{thm:cotilt-models}.

Consider now the complete cotorsion pair $(\mcX_C,\mcF_0)$ and an approximation sequence $0 \to X \overset{i}\to F_X \to C_X \to 0$ with $F_X \in \mcF_0$ and $C_X \in \mcX_C$ in $\ModR$. Since $X \in \mcX_C$, we have $F_X \in \mcX_C \cap \mcF_0 = \Prod C$ by Proposition~\ref{prop:cotilt-cot-pair}. Thus, since $Z \in \mcG \subseteq \Dtge0_C$ (see Notation~\ref{nota:cotilt-t-str}), we have $\Hom_{\Der R}(Z,\Sigma F_X) = 0$ and by Corollary~\ref{cor:compute-in-DerR-cotilt} also $\Hom_{\Htp R}(Z,\Sigma F_X) = 0$. This in particular means that the morphism $i\dd X \to F_X$ factors through $f\dd X \to Y$ and $f$ is a monomorphism in $\ModR$. Consequently up to isomorphism in $\mcG$ we have $Z \in \mcG \cap \ModR = \mcX_C$, proving the claim.

A similar argument works for $X,Y \in \Prod C$. The difference in this case is that $i\dd X \to F_X$ splits as $\Ext^1_R(C_X,X) = 0$, so $f\dd X \to Y$ splits as well and up to isomorphism $Z \in \Prod C$.

Since $\Prod C \subseteq \mcX_C$, it remains to prove that $\Prod C$ is cogenerating in $\mcG$. Given $X \in \mcG$, we can assume by Lemma~\ref{lem:right-aisle-cotilt} that $X$ as an object of $\Cpx R$ is of the form
\[ \cdots \la 0 \la 0 \la X^0 \mapr{\dif^0} X^1 \la X^2 \la \cdots \]
with $X^i \in \Prod C$ for all $i \ge 0$. Let us consider the triangle corresponding to the brutal truncation of $X$ at the differential $\dif^0\dd X^0 \to X^1$:
\[ X' \la X \la X^0 \la \Sigma X'. \]
In particular, the suspension $\Sigma X'$ is of the form
\[ \Sigma X'\dd \quad \cdots \la 0 \la 0 \la X^1 \la X^2 \la X^3 \la \cdots \]
with each $X^i$ in cohomological degree $i-1$. Since both $X$ and $\Sigma X'$ are then cofibrant \wrt the cotilting model structure on $\Cpx R$, we have
\begin{multline*}
\Hom_{\Der R}(\Sigma X',\Sigma^i C) \cong \Hom_{\Htp R}(\Sigma X',\Sigma^i C) \cong \\
\cong \Hom_{\Htp R}(X,\Sigma^{i-1} C) \cong \Hom_{\Der R}(X,\Sigma^{i-1} C) = 0
\end{multline*}
for all $i \in \bbZ \setminus \{0,1\}$. Clearly also $\Hom_{\Der R}(\Sigma X',\Sigma C) = 0$. Hence $\Sigma X' \in \mcG$ and, invoking Proposition~\ref{prop:heart}(2), we obtain a short exact sequence
\begin{equation} \label{eqn:co-t-str-trunc2}
0 \la X \la X^0 \la \Sigma X' \la 0
\end{equation}
in $\mcG$ with $X^0 \in \Prod C$.
\end{proof}

Next we establish a certain analogue of Lemma~\ref{lem:way-out} where the roles of $\Der R$ and $\Der\mcG$ are switched.

\begin{lem} \label{lem:heart-cohomology}
Let $R$, $C$, $\mcX_C$ and $\mcG$ be as above and let $n$ be the injective dimension of $C$ in $\ModR$. Then for each $X \in \mcG \subseteq \Der R$ we have $H^i_R(X) = 0$ for all $i<0$ and $i>n$.
\end{lem}

\begin{proof}
This quickly follows by combining the description of $\mcG$ in Definition~\ref{defn:tilted-catg} with the existence of an exact sequence as in Definition~\ref{defn:big-cotilt}(BC3) with $r = n$. Such a sequence exists by~\cite[Proposition 3.5]{Baz04}.
\end{proof}

Now we can prove that $\mcX_C$ is a dual blueprint for $\bbD_\mcG$.

\begin{prop} \label{prop:X_C-dual-blueprint}
Let $R$, $C$, $\mcX_C$ and $\mcG$ be as above. Then $\mcX_C$ is a dual blueprint for $\bbD_\mcG$.
\end{prop}

\begin{proof}
Although we have not proved yet that $\mcG$ is a Grothendieck category, the proof of Lemma~\ref{lem:functor-for-free} is still valid to show that both $\Prod C$ and $\mcX_C$ are functorially coresolving in $\mcG$. Indeed, since $\Prod C \subseteq \mcG$, arbitrary products of copies of $C$ exist in $\mcG$. Thus, for each $X \in \mcG$ we have the functorial monomorphism $X \to C^{\Hom_\mcG(X,C)}$.

It remains to prove that the $\mcX_C$-coresolution dimension of objects of $\mcG$ is uniformly bounded. Let $n$ be the injective dimension of $C$ in $\ModR$. Let us denote by $\mcT_j$, for $j \in \{0,1, \dots, n\}$, the class
\[ \mcT_j = \{ X \in \mcG \mid H^i_R(X) = 0 \textrm{ for all } i < 0 \textrm{ and } i > j \}. \]
Then, up to closing $\mcX_C$ under isomorphisms in $\mcG$, we have the following:
\[ \mcX_C = \mcT_0 \subseteq \mcT_1 \subseteq \cdots \subseteq \mcT_n = \mcG. \]

Suppose that $X \in \mcT_j$ for some $j > 0$ and consider the exact sequence $0 \to X \to X^0 \to \Sigma X' \to 0$ in $\mcG$ from $(\ref{eqn:co-t-str-trunc2})$ in the proof of Lemma~\ref{lem:X_C-coresolving}. It follows from the construction of~$(\ref{eqn:co-t-str-trunc2})$ that $\Sigma X' \in \mcT_{j-1}$. In particular $\Sigma X' \in \mcX_C$ if $j = 1$. Hence the coresolution dimension of $X \in \mcG = \mcT_n$ \wrt $\mcX_C$ is bounded by $n$.
\end{proof}

Now we can prove the main result of the section.

\begin{thm} \label{thm:equiv-derivators}
Let $R$ be a ring, $C \in \ModR$ be a big cotilting module and $\mcG$ be the corresponding tilted abelian category (Definition~\ref{defn:tilted-catg}).

Then the prederivators $\bbD_R$ and $\bbD_\mcG$ (see Construction~\ref{constr:der-abel}) are equivalent and both are strong stable derivators.
In particular, the inclusion $\mcG \subseteq \Der R$ extends to a triangle equivalence $F\dd \Der\mcG \to \Der R$.
\end{thm}

\begin{proof}
Let $\mcX_C \subseteq \ModR$ be the cotilting class. Propositions~\ref{prop:blueprint}, \ref{prop:X_C-blueprint} and~\ref{prop:X_C-dual-blueprint} together with Remark~\ref{rem:blueprint-dual} imply that the morphisms of prederivators
\begin{equation} \label{eqn:equiv-derivators}
\bbD_R \longleftarrow \bbD_{\mcX_C} \la \bbD_\mcG
\end{equation}
induced by the inclusions of $\mcX_C$ in $\ModR$ and $\mcG$, respectively, are equivalences. Since $\bbD_R$ is a strong stable derivator by Proposition~\ref{prop:stable-deriv-Grothendieck}(1), so are $\bbD_{\mcX_C}$ and $\bbD_\mcG$. In particular, $\bbD_\mcG(I)$ is a legitimate category in the set theoretic universe which we started with for each $I \in \Cat$.

The last statement follows by inspecting $(\ref{eqn:equiv-derivators})$ evaluated at $\unit \in \Cat$ and by Proposition~\ref{prop:stable-deriv-Grothendieck}(3).
\end{proof}

\section{Properties of the heart of a cotilting $t$-structure}
\label{sec:heart}

It is time to prove that the abelian category $\mcG$ tilted with respect to a big cotilting module $C$ is a Grothendieck category with a classical tilting object. It will turn out that the tilting object is necessarily finitely presentable. Finally, we will show that there is a bijective correspondence between certain equivalence classes of module categories with a big cotilting object and equivalence classes of Grothendieck categories with a classical tilting object.

\subsection{The heart is a Grothendieck category}
\label{subsec:heart-Grothendieck}

We start with an easy consequence of the previous results.

\begin{lem} \label{lem:enough-inj}
Let $C$ be a big cotilting $R$-module and $\mcG$ the tilted category. Then $C \in \mcG$ is an injective cogenerator for $\mcG$. Moreover, an object $W \in \mcG$ is injective if and only if, up to isomorphism, $W \in \Prod C$.
\end{lem}

\begin{proof}
Recall that products of copies of $C$ exist in $\mcG$ and agree with the products in $\Der R$. We know from Lemma~\ref{lem:X_C-coresolving} that $C$ is a cogenerator for $\mcG$. Given an arbitrary $X \in \mcG$, Proposition~\ref{prop:heart}(3) and Definition~\ref{defn:tilted-catg} imply that
\[ \Ext^1_\mcG(X,C) \cong \Hom_{\Der R}(X,\Sigma C) = 0. \]
Thus, $C$ is injective. Finally, $W$ is injective \iff the embedding $W \to C^{\Hom_\mcG(W,C)}$ splits \iff $W \in \Prod C$.
\end{proof}

Now we can prove that $\mcG$ is a Grothendieck category with a tilting object, which establishes a converse of Theorem~\ref{thm:cotilt-from-tilt} and generalizes~\cite[Theorem 4.2]{CGM07} for cotilting modules of injective dimension $>1$. We shall use a trick with a functor category suggested by Ivo Herzog.

\begin{thm} \label{thm:heart-Grothendieck}
Let $R$ be a ring, $C \in \ModR$ be a big cotilting module and $\mcG$ be the corresponding tilted abelian category (Definition~\ref{defn:tilted-catg}). Then $\mcG$ is a Grothendieck category.

Moreover, $\mcG$ admits a classical tilting object $T \in \mcG$ (see Definition~\ref{defn:tilt} and Proposition~\ref{prop:tilt-char}) \st $\End_{\mcG}(T) \cong R$ and $\Ext^{n+1}_\mcG(T,-) \equiv 0$, where $n$ is the injective dimension of $C$ in $\ModR$. In particular we can choose $T = R$, considered as an object of $\mcG$.
\end{thm}

\begin{proof}
We denote by $\modR\op$ the full subcategory of $\ModR\op$ consisting of finitely presented left $R$-modules, and by $\mcA = (\modR\op,\Ab)$ the category of all additive functors $\modR\op \to \Ab$. Then $\mcA$ is a locally coherent Grothendieck category and we have a fully faithful right exact functor
\[ T\dd \ModR \la \mcA, \quad M \longmapsto M \otimes_R -, \]
which preserves products, coproducts and sends pure injective modules to injective objects of $\mcA$; see~\cite[Proposition 1.2]{GrJe81}.

Since $C$ is a pure injective module itself by~\cite[Theorem 13]{St06}, $T(C) \in \mcA$ is injective. Thus, the class
\[ \mcT_C = \big\{ A \in \mcA \mid \Hom_\mcA\big(A,T(C)\big) = 0 \big\} \]
is a hereditary torsion class in $\mcA$. That is, $\mcT_C$ is closed under subfunctors, quotients, extensions and coproducts in $\mcA$.

Hence we can form the Gabriel quotient $Q\dd \mcA \to \mcA/\mcT_C$, where $\mcA/\mcT_C$ is by definition $\mcA[\we\inv]$ and $\we$ is the class of all morphisms in $\mcA$ whose kernel and cokernel belong to $\mcT_C$. We shall denote $\mcA/\mcT_C$ by $\mcG'$; it is an abelian category and $Q$ is an exact functor by~\cite[\S I.3]{GZ67}. In fact, it is well known that $\mcG'$ is a Grothendieck category and $Q$ admits a fully faithful right adjoint $H\dd \mcG' \to \mcA$, whose essential image is equal to
\[ \Img H = \{ X \in \mcA \mid \Hom_\mcA(A,X) = 0 = \Ext^1_\mcA(A,X) \textrm{ for all } A \in \mcT_C \}. \]
We refer to~\cite[\S11.1.1]{Pre09} and~\cite{Pop73,Sten75} for details.

As a right adjoint to an exact functor, $H$ must send injective objects of $\mcG'$ to injective objects of $\mcA$. Since $H$ is also fully faithful, it follows that $X \in \mcG'$ is injective \iff $H(X) \in \mcA$ is injective. Taking into account the definition of $\mcT_C$, we see that $X \in \mcG'$ is injective \iff $H(X) \in \Prod T(C)$; see also~\cite[Theorem 11.1.5]{Pre09}. To summarize, the composition
\[ \ModR \mapr{T} \mcA \mapr{Q} \mcG' \]
restricts to an equivalence between $\Prod C$ and the full subcategory $\mcI' \subseteq \mcG'$ of injective objects of $\mcG'$.

Denoting by $\mcI \subseteq \mcG$ the full subcategory of injective objects of $\mcG$, Lemma~\ref{lem:enough-inj} provides us with an equivalence $\Prod C \to \mcI$.

The composed equivalence $\mcI \to \mcI'$ extends to and equivalence $\mcG \to \mcG'$. Indeed, it is a standard observation that an abelian category with enough injective objects is determined by its full subcategory of injective objects up to equivalence. We refer for instance to~\cite[Proposition IV.1.2]{ARS97} for a dual argument for abelian categories with enough projective objects. Thus, $\mcG$ is a Grothendieck category.

If we put $T = R \in \mcG$, then $T$ satisfies (T1)--(T3) of Proposition~\ref{prop:tilt-char} since $R$ satisfies the same properties in $\Der R$. Here we use Theorem~\ref{thm:equiv-derivators}. Therefore, $T \in \mcG$ is a tilting object whose endomorphism ring is isomorphic to $R$. The extra property that $\Ext^{n+1}_\mcG(T,-) \equiv 0$ is a consequence of the isomorphisms
\[ \Ext^{n+1}_\mcG(T,X) \cong \Hom_{\Der G}(T,\Sigma^{n+1}X) \cong \Hom_{\Der R}(R,\Sigma^{n+1}X) \cong H_R^{n+1}(X) \]
and of Lemma~\ref{lem:heart-cohomology}.
\end{proof}

\begin{rem} \label{rem:left-adj-to-Hom(T,-)}
In the proof of Theorem~\ref{thm:QE} we established the existence of a left adjoint to $\Hom_\mcG(T,-)\dd \Cpx\mcG \to \Cpx R$. Now we can obtain more information on this adjoint functor.

Using the notation from the proof of Theorem~\ref{thm:heart-Grothendieck} and identifying $\mcG'$ with $\mcG$ via the equivalence, we have a diagram of categories and functors
\[ \xymatrix@1{\;\ModR\; \ar@<.7ex>[rr]^-T && \;\mcA\; \ar@<.7ex>[rr]^-Q \ar@<.7ex>[ll]^-E && \;\mcG. \ar@<.7ex>[ll]^-H} \]
The functor $E$ is the right adjoint to $T$ (which exists by the special adjoint functor theorem) and it acts by evaluating a functor $F\dd \modR\op \to \Ab$ at $R$.

It is not difficult to convince oneself that $E \circ H$ and $\Hom_\mcG(T,-)$ are equivalent functors $\mcG \to \ModR$. Indeed, they are both left exact and their restrictions to the class of injective objects $\mcI \subseteq \mcG$ are equivalent. Thus, $Q\circ T$ is a left adjoint to $\Hom_\mcG(T,-)$, both as a functor $\ModR \to \mcG$ and as a functor $\Cpx R \to \Cpx\mcG$.
\end{rem}

\begin{rem} \label{rem:generators-of-the-heart}
As a Grothendieck category, $\mcG$ must have a small generating set. In order to construct one, let $T=R$ be as in Theorem~\ref{thm:heart-Grothendieck}. Then one can show that the set
\[ \mcS = \{ G \in \mcG \mid G \textrm{ is a subobject of } T^r \textrm{ for some } r \in \bbN \} \]
is generating; see for instance \cite[Lemma 3.5]{CGM07}.
\end{rem}

\subsection{Homotopy finitely presentable objects}
\label{subsec:htpy-finite}

Let us now turn our attention to the ways in which we express smallness in the abelian and triangulated context. As in~\cite{AR94,GU71}, we call an object $Z$ of a cocomplete category $\mcG$ \emph{finitely presentable} if given any direct system $(Z_i \mid i \in I)$ in $\mcG$, the canonical morphism of abelian groups
\[ \li_I \Hom_\mcG(Z,X_i) \la \Hom_\mcG(Z, \li_I X_i) \]
is an isomorphism. If $\mcG = \ModR$, finitely presentable objects are well known to coincide with finitely presented modules in the classical sense. If on the other hand $\mcT$ is triangulated (for instance $\mcT = \Der\mcG$ for a Grothendieck category $\mcG$), we have the notion of compactness (Definition~\ref{defn:comp-gen}). In order to link the two previous notions under certain assumptions, we introduce the following concept:

\begin{defn} \label{defn:htpy-fp}
Let $\bbD\dd \Cat\op \to \CAT$ be a stable derivator and let $\mcT = \bbD(\unit)$. An object $Z \in \mcT$ is called \emph{homotopy finitely presentable} if, given any directed poset $I$ and a coherent diagram $X \in \bbD(I)$, the canonical map
\[ \li_I \Hom_\mcT(Z,X_i) \la \Hom_\mcT(Z, \hocolim_I X) \]
is bijective.

To give more details on the construction of the morphism, consider $i \in I$, the embedding $i\dd \unit \to I$ and the projection $p\dd I \to \unit$. We can apply $i^*$ to the unit of adjunction $\eta\dd X \to p^*(\hocolim_I X)$. Since $p\circ i = \id_\unit$ and $i^*\circ p^* = \id_\mcT$, we obtain a morphism $\ep_i\dd X_i \to \hocolim_I X$ in $\mcT$. It is straightforward to check that the morphisms $\ep_i$ actually form a cocone in $\mcT$ from the direct system $(X_i \mid i \in I)$ to $\hocolim_I X$. It remains to apply $\Hom_\mcT(Z,-)\dd \mcT \to \Set$ and construct the colimit morphism in $\Set$.
\end{defn}

If $\bbD$ is a derivator for a ring, we can show that this property coincides with compactness.

\begin{prop} \label{prop:htpy-fp-vs-compact}
Let $R$ be a ring and $\bbD_R$ be the strong stable derivator from Construction~\ref{constr:der-abel}. Then $Z \in \bbD_R(\unit) = \Der R$ is homotopy finitely presentable \iff it $Z$ is compact in $\Der R$.
\end{prop}

\begin{proof}
Suppose first that $Z \in \Der R$ is compact. It follows from the proof~\cite[Propositions 6.2 and 6.3]{Rick89} that, up to isomorphism, $Z$ must be a bounded complex of finitely generated projective modules.

Suppose now that we have a directed set $I$ and an object $X \in \bbD_R(I)$. We can view $X$ as an object of $\Cpx R^I$. We claim that $\hocolim_I X$ is simply the colimit of $X$ in $\Cpx R$. Indeed, the colimit functor $\li\dd \Cpx R^I \to \Cpx R$ and the constant diagram functor $\Cpx R \to \Cpx R^I$ are both exact and form an adjoint pair. Hence they induce an adjoint pair of functors between the corresponding derived categories and $\li\dd \bbD(I) \to \bbD(\unit)$ becomes a left adjoint to $p^*$ (using the notation of Definition~\ref{defn:htpy-fp}). Since a left adjoint is unique up to equivalence, we have $\li \cong \hocolim_I$, proving the claim.

Since $Z$ is bounded and its components are finitely generated and projective, the canonical morphism in $\Cpx\Ab$ (see Notation~\ref{nota:total-hom})
\[ \li_I \HOM_R(Z,X_i) \la \HOM_R(Z, \li X) \]
is clearly an isomorphism. Passing to the zero cohomology, we obtain the isomorphism
\[ \li_I \Hom_{\Htp R}(Z,X_i) \la \Hom_{\Htp R}(Z,\hocolim_I X), \]
and also the isomorphism of the corresponding homomorphism groups in $\Der R$. It follows that $Z$ is homotopy finitely presentable.

Conversely, let $Z$ be homotopy finitely presentable and suppose that $(X_j \mid j \in J)$ is a small collection of objects of $\Der R$. Denote by $I$ the poset of all finite subsets of $J$ ordered by inclusion. It is easy to construct a coherent diagram $X \in \bbD(I)$ \st $i^*(X) = \bigoplus_{j \in i} X_j$, $\hocolim_I = \bigoplus_{j \in J} X_j$ and the morphisms $\ep_i\dd i^*(X) \to \hocolim_I X$ from Definition~\ref{defn:htpy-fp} are the canonical split inclusions. The isomorphism $\li_I \Hom_{\Der R}\big(Z, i^*(X)\big) \cong \Hom_{\Der R}(Z, \hocolim_I X)$ then precisely amounts to the defining condition of compactness in Definition~\ref{defn:comp-gen}.
\end{proof}

The equivalence of derivators from Theorem~\ref{thm:equiv-derivators} allows us to conclude that objects in $\mcG$ which are compact in $\Der\mcG$ are finitely presentable in $\mcG$. This in particular applies to the classical tilting object from Theorem~\ref{thm:heart-Grothendieck}, see also Remark~\ref{rem:classical}.

\begin{thm} \label{thm:comp-gen-to-fp}
Let $C$ be a big cotilting $R$-module and $\mcG$ be the corresponding tilted Grothendieck category. Suppose that $Z \in \mcG$ is compact in $\Der\mcG$ (for example if $Z = R \in \mcG$ is the tilting object). Then:
\begin{enumerate}
\item The functor $\Ext^m_\mcG(Z,-)\dd \mcG \to \Ab$ preserves direct limits for all $m \ge 0$. In particular, $Z$ is finitely presentable in $\mcG$.
\item $\Ext^m_\mcG(Z,-) \equiv 0$ for $m \gg 0$.
\end{enumerate}
\end{thm}

\begin{proof}
Since $\bbD_\mcG$ is equivalent to $\bbD_R$, the object $\Sigma^{-m} Z$ is homotopy finitely presentable in $\bbD_\mcG(\unit)$ by Proposition~\ref{prop:htpy-fp-vs-compact}. Suppose now that $X = (X_i \mid i \in I)$ is a direct system in $\mcG$. We can view $X$ as an object of $\bbD(I) = \Der{\mcG^I}$. For the same reason as in the proof of Proposition~\ref{prop:htpy-fp-vs-compact} we can identify $\hocolim_I X$ with the ordinary colimit $\li_I X_i$ in $\mcG$. Part (1) then follows from the isomorphisms
\begin{multline*}
\li_I \Ext^m_\mcG(Z,X_i) \cong \li_I \Hom_{\Der\mcG}(\Sigma^{-m} Z, X_i) \cong \\
\cong \Hom_{\Der\mcG}(\Sigma^{-m} Z,\hocolim_I X) \cong \Ext^m_\mcG(Z,\li X_i).
\end{multline*}

For part (2), we use Lemma~\ref{lem:heart-cohomology} and the equivalence $F\dd \Der\mcG \to \Der R$ from Theorem~\ref{thm:equiv-derivators}. Since $F(Z)$ is isomorphic to a bounded complex with projective components, it follows that there exists $N \in \bbZ$ \st $\Hom_{\Der R}\big(F(Z),\Sigma^m X\big) = 0$ for all $X \in \mcG$ and $m \ge N$, as required.
\end{proof}

\subsection{A bijective correspondence}
\label{subsec:bijection}

We conclude the paper by making precise in what sense the constructions behind Theorems~\ref{thm:cotilt-from-tilt} and~\ref{thm:heart-Grothendieck} are inverse to each other. Consider two ordered triples $(\mcA_1, X_1, Y_1)$ and $(\mcA_2, X_2, Y_2)$ where $\mcA_i$ is a category and $X_i,Y_i \in \mcA_i$ are objects for $i = 1,2$. We say that $(\mcA_1, X_1, Y_1)$ and $(\mcA_2, X_2, Y_2)$ are \emph{equivalent} if there is an equivalence of categories $F\dd \mcA_1 \to \mcA_2$ \st $F(X_1) \cong X_2$ and $F(Y_1) \cong F(Y_2)$.

\begin{thm} \label{thm:bijection}
For every $n \ge 0$, there is a bijective correspondence between:
\begin{enumerate}
\item Equivalence classes of $(\mcG,T,W)$ where $\mcG$ is a Grothendieck category, $T \in \mcG$ is a classical tilting object \st $\Ext^{n+1}_\mcG(T,-) \equiv 0$, and $W \in \mcG$ is an injective cogenerator.
\item Equivalence classes of $(\ModR,R,C)$, where $R$ is a ring and $C \in \ModR$ is a big $n$-cotilting module.
\end{enumerate}

Given $(\mcG,T,W)$, we assign to it $(\ModR,R,C)$ \st $R = \End_\mcG(T)$ and $C = \Hom_\mcG(T,W)$.

Starting with $(\ModR,R,C)$, we take for $\mcG$ the tilted Grothendieck category corresponding to $C$ (Definition~\ref{defn:tilted-catg}), and put $T = R$ and $W = C$.
\end{thm}

\begin{proof}
The assignments are well defined by Theorems~\ref{thm:cotilt-from-tilt} and~\ref{thm:heart-Grothendieck}. The fact that they are mutually inverse follows from the descriptions of the corresponding derived equivalences in Theorems~\ref{thm:QE} and~\ref{thm:equiv-derivators}.
\end{proof}

\bibliographystyle{alpha}
\bibliography{references}

\end{document}

%% file: macros.tex
\renewcommand{\iff}{if and only if }
\newcommand{\st}{such that }
\newcommand{\wrt}{with respect to }

\newcommand{\la}{\longrightarrow}
\newcommand{\mapr}[1]{\overset{#1}\longrightarrow}

\newcommand{\dd}{\colon}

\newcommand{\ep}{\varepsilon}

\newcommand{\dif}{\partial}

\newcommand{\CpxG}{{\Cpx\mcG}}
\newcommand{\HtpG}{{\Htp\mcG}}
\newcommand{\DerG}{{\Der\mcG}}

\newcommand{\id}{\mathrm{id}}
\newcommand{\inv}{^{-1}}

\newcommand{\BC}[1]{\mathbb{BC}({#1})}

\newcommand{\OO}{\mathcal{O}}              
\newcommand{\PP}[2]{\mathbb{P}^{#1}_{#2}}  

\newcommand{\Hom}{\operatorname{Hom}}
\newcommand{\HOM}{\mathcal{H}\mathnormal{om}}
\newcommand{\End}{\operatorname{End}}

\newcommand{\Ext}{\operatorname{Ext}}

\newcommand{\Ker}{\operatorname{Ker}}
\newcommand{\Img}{\operatorname{Im}}
\newcommand{\Coker}{\operatorname{Coker}}

\newcommand{\resdim}{\operatorname{res.dim}}

\newcommand{\li}{\varinjlim}
\newcommand{\holim}{\operatorname{holim}}
\newcommand{\hocolim}{\operatorname{hocolim}}

\newcommand{\bbD}{\mathbb{D}}
\newcommand{\bbN}{\mathbb{N}}
\newcommand{\bbZ}{\mathbb{Z}}
\newcommand{\bbQ}{\mathbb{Q}}

\newcommand{\mcA}{\mathcal{A}}
\newcommand{\mcB}{\mathcal{B}}
\newcommand{\mcC}{\mathcal{C}}
\newcommand{\mcD}{\mathcal{D}}
\newcommand{\mcE}{\mathcal{E}}
\newcommand{\mcF}{\mathcal{F}}
\newcommand{\mcG}{\mathcal{G}}

\newcommand{\mcI}{\mathcal{I}}

\newcommand{\mcP}{\mathcal{P}}

\newcommand{\mcS}{\mathcal{S}}
\newcommand{\mcT}{\mathcal{T}}
\newcommand{\mcU}{\mathcal{U}} 
 
\newcommand{\mcX}{\mathcal{X}} 
\newcommand{\mcY}{\mathcal{Y}}

\newcommand{\ModR}{\mathrm{Mod}\textrm{-}R}

\newcommand{\modR}{\mathrm{mod}\textrm{-}R}

\newcommand{\Qco}[1]{\mathrm{Qcoh}({#1})}

\newcommand{\Cat}{\mathcal{C}\mathnormal{at}}
\newcommand{\CAT}{\mathcal{CAT}}
\newcommand{\unit}{\mathds{1}}

\newcommand{\Set}{\mathrm{Set}}
\newcommand{\Ab}{\mathrm{Ab}}

\newcommand{\op}{^\textrm{op}}

\newcommand{\Add}{\operatorname{Add}}
\newcommand{\Prod}{\operatorname{Prod}} 
\newcommand{\Der}[1]{\mathbf{D}({#1})}

\newcommand{\Htp}[1]{\mathbf{K}({#1})}
\newcommand{\Cpx}[1]{\mathbf{C}({#1})}

\newcommand{\cof}{\mathrm{Cof}}
\newcommand{\Cof}{\mathcal{C}}

\newcommand{\we}{\mathrm{W}}
\newcommand{\We}{\mathcal{W}}

\newcommand{\fib}{\mathrm{Fib}}
\newcommand{\Fib}{\mathcal{F}}

\newcommand{\Rder}[1]{\mathbf{R}{#1}}
\newcommand{\Lotimes}{\otimes^{\mathbf{L}}}
\newcommand{\RHom}{\Rder\Hom}

\newcommand{\Dtle}[1]{\mcD^{\le{#1}}}
\newcommand{\Dtge}[1]{\mcD^{\ge{#1}}}
\newcommand{\ttle}[1]{\tau^{\le{#1}}}
\newcommand{\ttge}[1]{\tau^{\ge{#1}}}

\theoremstyle{plain}
\newtheorem{thm}{Theorem}[section]
\newtheorem{lem}[thm]{Lemma}
\newtheorem{prop}[thm]{Proposition}
\newtheorem{cor}[thm]{Corollary}

\theoremstyle{definition}
\newtheorem{defn}[thm]{Definition}
\newtheorem{constr}[thm]{Construction}
\newtheorem{nota}[thm]{Notation}

\theoremstyle{remark}
\newtheorem{rem}[thm]{Remark}

\newtheorem{expl}[thm]{Example}

%% file: cotilt-derived-eq_v4.bbl
\begin{thebibliography}{ATJLSS03}

\bibitem[AB69]{AuBr69}
Maurice Auslander and Mark Bridger.
\newblock {\em Stable module theory}.
\newblock Memoirs of the American Mathematical Society, No. 94. American
  Mathematical Society, Providence, R.I., 1969.

\bibitem[AF92]{AF92}
Frank~W. Anderson and Kent~R. Fuller.
\newblock {\em Rings and categories of modules}, volume~13 of {\em Graduate
  Texts in Mathematics}.
\newblock Springer-Verlag, New York, second edition, 1992.

\bibitem[AHC01]{ACo01}
Lidia Angeleri~H{\"u}gel and Fl{\'a}vio~Ulhoa Coelho.
\newblock Infinitely generated tilting modules of finite projective dimension.
\newblock {\em Forum Math.}, 13(2):239--250, 2001.

\bibitem[AHHK07]{AHK07}
Lidia Angeleri~H{\"u}gel, Dieter Happel, and Henning Krause, editors.
\newblock {\em Handbook of tilting theory}, volume 332 of {\em London
  Mathematical Society Lecture Note Series}.
\newblock Cambridge University Press, Cambridge, 2007.

\bibitem[AHP{\v{S}}T14]{APST12}
Lidia Angeleri~H{\"u}gel, David Posp{\'{\i}}{\v{s}}il, Jan
  {\v{S}}{\v{t}}ov{\'{\i}}{\v{c}}ek, and Jan Trlifaj.
\newblock Tilting, cotilting, and spectra of commutative noetherian rings.
\newblock {\em Trans. Amer. Math. Soc.}, 366(7):3487--3517, 2014.

\bibitem[AHS14]{ASa13}
Lidia Angeleri~H{\"u}gel and Manuel Saor{\'{\i}}n.
\newblock {$t$}-structures and cotilting modules over commutative noetherian
  rings.
\newblock To appear in Math. Z. (published on-line,
  doi:10.1007/s00209-014-1281-y), 2014.

\bibitem[AR91]{AR91}
Maurice Auslander and Idun Reiten.
\newblock Applications of contravariantly finite subcategories.
\newblock {\em Adv. Math.}, 86(1):111--152, 1991.

\bibitem[AR94]{AR94}
Ji{\v{r}}{\'{\i}} Ad{\'a}mek and Ji{\v{r}}{\'{\i}} Rosick{\'y}.
\newblock {\em Locally presentable and accessible categories}, volume 189 of
  {\em London Mathematical Society Lecture Note Series}.
\newblock Cambridge University Press, Cambridge, 1994.

\bibitem[ARS97]{ARS97}
Maurice Auslander, Idun Reiten, and Sverre~O. Smal{\o}.
\newblock {\em Representation theory of {A}rtin algebras}, volume~36 of {\em
  Cambridge Studies in Advanced Mathematics}.
\newblock Cambridge University Press, Cambridge, 1997.
\newblock Corrected reprint of the 1995 original.

\bibitem[ATJLS10]{AJSa10}
Leovigildo Alonso~Tarr{\'{\i}}o, Ana Jerem{\'{\i}}as~L{\'o}pez, and Manuel
  Saor{\'{\i}}n.
\newblock Compactly generated {$t$}-structures on the derived category of a
  {N}oetherian ring.
\newblock {\em J. Algebra}, 324(3):313--346, 2010.

\bibitem[ATJLSS00]{AJS00}
Leovigildo Alonso~Tarr{\'{\i}}o, Ana Jerem{\'{\i}}as~L{\'o}pez, and
  Mar{\'{\i}}a~Jos{\'e} Souto~Salorio.
\newblock Localization in categories of complexes and unbounded resolutions.
\newblock {\em Canad. J. Math.}, 52(2):225--247, 2000.

\bibitem[ATJLSS03]{AJS03}
Leovigildo Alonso~Tarr{\'{\i}}o, Ana Jerem{\'{\i}}as~L{\'o}pez, and
  Mar{\'{\i}}a~Jos{\'e} Souto~Salorio.
\newblock Construction of {$t$}-structures and equivalences of derived
  categories.
\newblock {\em Trans. Amer. Math. Soc.}, 355(6):2523--2543 (electronic), 2003.

\bibitem[Baz04]{Baz04}
Silvana Bazzoni.
\newblock A characterization of {$n$}-cotilting and {$n$}-tilting modules.
\newblock {\em J. Algebra}, 273(1):359--372, 2004.

\bibitem[BBD82]{BBD82}
A.~A. Be{\u\i}linson, J.~Bernstein, and P.~Deligne.
\newblock Faisceaux pervers.
\newblock In {\em Analysis and topology on singular spaces, {I} ({L}uminy,
  1981)}, volume 100 of {\em Ast\'erisque}, pages 5--171. Soc. Math. France,
  Paris, 1982.

\bibitem[BC76]{BC76}
Edgar~H. Brown, Jr. and Michael Comenetz.
\newblock Pontrjagin duality for generalized homology and cohomology theories.
\newblock {\em Amer. J. Math.}, 98(1):1--27, 1976.

\bibitem[Bec14]{Beck12}
Hanno Becker.
\newblock Models for singularity categories.
\newblock {\em Adv. Math.}, 254:187--232, 2014.

\bibitem[Be{\u\i}78]{Bei78}
A.~A. Be{\u\i}linson.
\newblock Coherent sheaves on {${\bf P}^{n}$} and problems in linear algebra.
\newblock {\em Funktsional. Anal. i Prilozhen.}, 12(3):68--69, 1978.

\bibitem[BK89]{BoKa89}
A.~I. Bondal and M.~M. Kapranov.
\newblock Representable functors, {S}erre functors, and reconstructions.
\newblock {\em Izv. Akad. Nauk SSSR Ser. Mat.}, 53(6):1183--1205, 1337, 1989.

\bibitem[BN93]{BN93}
Marcel B{\"o}kstedt and Amnon Neeman.
\newblock Homotopy limits in triangulated categories.
\newblock {\em Compositio Math.}, 86(2):209--234, 1993.

\bibitem[Bor94]{Bor94-vol1}
Francis Borceux.
\newblock {\em Handbook of categorical algebra. 1}, volume~50 of {\em
  Encyclopedia of Mathematics and its Applications}.
\newblock Cambridge University Press, Cambridge, 1994.
\newblock Basic category theory.

\bibitem[B{\"u}h10]{Bu10}
Theo B{\"u}hler.
\newblock Exact categories.
\newblock {\em Expo. Math.}, 28(1):1--69, 2010.

\bibitem[BvdB03]{BvdB03}
A.~Bondal and M.~van~den Bergh.
\newblock Generators and representability of functors in commutative and
  noncommutative geometry.
\newblock {\em Mosc. Math. J.}, 3(1):1--36, 258, 2003.

\bibitem[CE56]{CE56}
Henri Cartan and Samuel Eilenberg.
\newblock {\em Homological algebra}.
\newblock Princeton University Press, Princeton, N. J., 1956.

\bibitem[CGM07]{CGM07}
Riccardo Colpi, Enrico Gregorio, and Francesca Mantese.
\newblock On the heart of a faithful torsion theory.
\newblock {\em J. Algebra}, 307(2):841--863, 2007.

\bibitem[Cis03]{Cis03}
Denis-Charles Cisinski.
\newblock Images directes cohomologiques dans les cat\'egories de mod\`eles.
\newblock {\em Ann. Math. Blaise Pascal}, 10(2):195--244, 2003.

\bibitem[CMT10]{CMT10}
Riccardo Colpi, Francesca Mantese, and Alberto Tonolo.
\newblock Cotorsion pairs, torsion pairs, and {$\Sigma$}-pure-injective
  cotilting modules.
\newblock {\em J. Pure Appl. Algebra}, 214(5):519--525, 2010.

\bibitem[CN08]{CisNee08}
Denis-Charles Cisinski and Amnon Neeman.
\newblock Additivity for derivator {$K$}-theory.
\newblock {\em Adv. Math.}, 217(4):1381--1475, 2008.

\bibitem[CPS86]{CPS86}
E.~Cline, B.~Parshall, and L.~Scott.
\newblock Derived categories and {M}orita theory.
\newblock {\em J. Algebra}, 104(2):397--409, 1986.

\bibitem[Dug01]{Dug01}
Daniel Dugger.
\newblock Combinatorial model categories have presentations.
\newblock {\em Adv. Math.}, 164(1):177--201, 2001.

\bibitem[EGAPT12]{EGPT12}
Sergio Estrada, Pedro~A. Guil~Asensio, Mike Prest, and Jan Trlifaj.
\newblock Model category structures arising from {D}rinfeld vector bundles.
\newblock {\em Adv. Math.}, 231(3-4):1417--1438, 2012.

\bibitem[EJ11]{EJv2}
Edgar~E. Enochs and Overtoun M.~G. Jenda.
\newblock {\em Relative homological algebra. {V}olume 2}, volume~54 of {\em de
  Gruyter Expositions in Mathematics}.
\newblock Walter de Gruyter GmbH \& Co. KG, Berlin, 2011.

\bibitem[ET01]{ET01}
Paul~C. Eklof and Jan Trlifaj.
\newblock How to make {E}xt vanish.
\newblock {\em Bull. London Math. Soc.}, 33(1):41--51, 2001.

\bibitem[Fre64]{Fr64}
Peter Freyd.
\newblock {\em Abelian categories. {A}n introduction to the theory of
  functors}.
\newblock Harper's Series in Modern Mathematics. Harper \& Row Publishers, New
  York, 1964.

\bibitem[Gil04]{Gil04}
James Gillespie.
\newblock The flat model structure on {${\rm Ch}(R)$}.
\newblock {\em Trans. Amer. Math. Soc.}, 356(8):3369--3390 (electronic), 2004.

\bibitem[Gil07]{Gil07}
James Gillespie.
\newblock Kaplansky classes and derived categories.
\newblock {\em Math. Z.}, 257(4):811--843, 2007.

\bibitem[Gil11]{Gil11}
James Gillespie.
\newblock Model structures on exact categories.
\newblock {\em J. Pure Appl. Algebra}, 215(12):2892--2902, 2011.

\bibitem[GJ81]{GrJe81}
L.~Gruson and C.~U. Jensen.
\newblock Dimensions cohomologiques reli\'ees aux foncteurs
  {$\varprojlim^{(i)}$}.
\newblock In {\em Paul {D}ubreil and {M}arie-{P}aule {M}alliavin {A}lgebra
  {S}eminar, 33rd {Y}ear ({P}aris, 1980)}, volume 867 of {\em Lecture Notes in
  Math.}, pages 234--294. Springer, Berlin, 1981.

\bibitem[Gro57]{Gro57}
Alexander Grothendieck.
\newblock Sur quelques points d'alg\`ebre homologique.
\newblock {\em T\^ohoku Math. J. (2)}, 9:119--221, 1957.

\bibitem[Gro91]{GrothDerivators}
A.~Grothendieck.
\newblock Les d\'erivateurs.
\newblock available at
  \url{http://www.math.jussieu.fr/~maltsin/groth/Derivateurs.html}, 1991.

\bibitem[Gro13]{Gr12}
Moritz Groth.
\newblock Derivators, pointed derivators and stable derivators.
\newblock {\em Algebr. Geom. Topol.}, 13(1):313--374, 2013.

\bibitem[GT06]{GT06}
R{\"u}diger G{\"o}bel and Jan Trlifaj.
\newblock {\em Approximations and endomorphism algebras of modules}, volume~41
  of {\em de Gruyter Expositions in Mathematics}.
\newblock Walter de Gruyter GmbH \& Co. KG, Berlin, 2006.

\bibitem[GU71]{GU71}
Peter Gabriel and Friedrich Ulmer.
\newblock {\em Lokal pr\"asentierbare {K}ategorien}.
\newblock Lecture Notes in Mathematics, Vol. 221. Springer-Verlag, Berlin,
  1971.

\bibitem[GZ67]{GZ67}
P.~Gabriel and M.~Zisman.
\newblock {\em Calculus of fractions and homotopy theory}.
\newblock Ergebnisse der Mathematik und ihrer Grenzgebiete, Band 35.
  Springer-Verlag New York, Inc., New York, 1967.

\bibitem[Hap87]{Hap87}
Dieter Happel.
\newblock On the derived category of a finite-dimensional algebra.
\newblock {\em Comment. Math. Helv.}, 62(3):339--389, 1987.

\bibitem[Har66]{Hart66}
Robin Hartshorne.
\newblock {\em Residues and duality}.
\newblock Lecture notes of a seminar on the work of A. Grothendieck, given at
  Harvard 1963/64. With an appendix by P. Deligne. Lecture Notes in
  Mathematics, No. 20. Springer-Verlag, Berlin, 1966.

\bibitem[Hir03]{Hir03}
Philip~S. Hirschhorn.
\newblock {\em Model categories and their localizations}, volume~99 of {\em
  Mathematical Surveys and Monographs}.
\newblock American Mathematical Society, Providence, RI, 2003.

\bibitem[Hov99]{Hov99}
Mark Hovey.
\newblock {\em Model categories}, volume~63 of {\em Mathematical Surveys and
  Monographs}.
\newblock American Mathematical Society, Providence, RI, 1999.

\bibitem[Hov02]{Hov02}
Mark Hovey.
\newblock Cotorsion pairs, model category structures, and representation
  theory.
\newblock {\em Math. Z.}, 241(3):553--592, 2002.

\bibitem[HRS96]{HRS96}
Dieter Happel, Idun Reiten, and Sverre~O. Smal{\o}.
\newblock Tilting in abelian categories and quasitilted algebras.
\newblock {\em Mem. Amer. Math. Soc.}, 120(575):viii+ 88, 1996.

\bibitem[Kel90]{Kel90}
Bernhard Keller.
\newblock Chain complexes and stable categories.
\newblock {\em Manuscripta Math.}, 67(4):379--417, 1990.

\bibitem[Kel94]{Kel94}
Bernhard Keller.
\newblock Deriving {DG} categories.
\newblock {\em Ann. Sci. \'Ecole Norm. Sup. (4)}, 27(1):63--102, 1994.

\bibitem[KL06]{KrLe06}
Henning Krause and Jue Le.
\newblock The {A}uslander-{R}eiten formula for complexes of modules.
\newblock {\em Adv. Math.}, 207(1):133--148, 2006.

\bibitem[Kra05]{Kr05}
Henning Krause.
\newblock The stable derived category of a {N}oetherian scheme.
\newblock {\em Compos. Math.}, 141(5):1128--1162, 2005.

\bibitem[Mal07]{Malt07-derivators}
Georges Maltsiniotis.
\newblock La {$K$}-th\'eorie d'un d\'erivateur triangul\'e.
\newblock In {\em Categories in algebra, geometry and mathematical physics},
  volume 431 of {\em Contemp. Math.}, pages 341--368. Amer. Math. Soc.,
  Providence, RI, 2007.

\bibitem[Mit64]{Mitch64}
Barry Mitchell.
\newblock The full imbedding theorem.
\newblock {\em Amer. J. Math.}, 86:619--637, 1964.

\bibitem[ML98]{McL2}
Saunders Mac~Lane.
\newblock {\em Categories for the working mathematician}, volume~5 of {\em
  Graduate Texts in Mathematics}.
\newblock Springer-Verlag, New York, second edition, 1998.

\bibitem[Nee01]{Nee01}
Amnon Neeman.
\newblock {\em Triangulated categories}, volume 148 of {\em Annals of
  Mathematics Studies}.
\newblock Princeton University Press, Princeton, NJ, 2001.

\bibitem[Pop73]{Pop73}
N.~Popescu.
\newblock {\em Abelian categories with applications to rings and modules}.
\newblock Academic Press, London, 1973.
\newblock London Mathematical Society Monographs, No. 3.

\bibitem[Pre09]{Pre09}
Mike Prest.
\newblock {\em Purity, spectra and localisation}, volume 121 of {\em
  Encyclopedia of Mathematics and its Applications}.
\newblock Cambridge University Press, Cambridge, 2009.

\bibitem[Qui67]{QHtp67}
Daniel~G. Quillen.
\newblock {\em Homotopical algebra}.
\newblock Lecture Notes in Mathematics, No. 43. Springer-Verlag, Berlin, 1967.

\bibitem[Ren09]{Ren09}
Olivier Renaudin.
\newblock Plongement de certaines th\'eories homotopiques de {Q}uillen dans les
  d\'erivateurs.
\newblock {\em J. Pure Appl. Algebra}, 213(10):1916--1935, 2009.

\bibitem[Ric89]{Rick89}
Jeremy Rickard.
\newblock Morita theory for derived categories.
\newblock {\em J. London Math. Soc. (2)}, 39(3):436--456, 1989.

\bibitem[Ros02]{Ro02}
J.~Rosick{\'y}.
\newblock Flat covers and factorizations.
\newblock {\em J. Algebra}, 253(1):1--13, 2002.

\bibitem[RR06]{RR06}
Idun Reiten and Claus~Michael Ringel.
\newblock Infinite dimensional representations of canonical algebras.
\newblock {\em Canad. J. Math.}, 58(1):180--224, 2006.

\bibitem[Sal79]{Sal79}
Luigi Salce.
\newblock Cotorsion theories for abelian groups.
\newblock In {\em Symposia {M}athematica, {V}ol. {XXIII} ({C}onf. {A}belian
  {G}roups and their {R}elationship to the {T}heory of {M}odules, {INDAM},
  {R}ome, 1977)}, pages 11--32. Academic Press, London, 1979.

\bibitem[{\v{S}}KT11]{SKT11}
Jan {\v{S}}{\v{t}}ov{\'{\i}}{\v{c}}ek, Otto Kerner, and Jan Trlifaj.
\newblock Tilting via torsion pairs and almost hereditary noetherian rings.
\newblock {\em J. Pure Appl. Algebra}, 215(9):2072--2085, 2011.

\bibitem[S{\v{S}}11]{SaoSt11}
Manuel Saor{\'{\i}}n and Jan {\v{S}}{\v{t}}ov{\'{\i}}{\v{c}}ek.
\newblock On exact categories and applications to triangulated adjoints and
  model structures.
\newblock {\em Adv. Math.}, 228(2):968--1007, 2011.

\bibitem[Ste75]{Sten75}
Bo~Stenstr{\"o}m.
\newblock {\em Rings of quotients}.
\newblock Springer-Verlag, New York, 1975.
\newblock Die Grundlehren der Mathematischen Wissenschaften, Band 217, An
  introduction to methods of ring theory.

\bibitem[{\v{S}}{\v{t}}o06]{St06}
Jan {\v{S}}{\v{t}}ov{\'{\i}}{\v{c}}ek.
\newblock All {$n$}-cotilting modules are pure-injective.
\newblock {\em Proc. Amer. Math. Soc.}, 134(7):1891--1897 (electronic), 2006.

\bibitem[{\v{S}}{\v{t}}o13]{St13-deconstr}
Jan {\v{S}}{\v{t}}ov{\'{\i}}{\v{c}}ek.
\newblock Deconstructibility and the {H}ill lemma in {G}rothendieck categories.
\newblock {\em Forum Math.}, 25(1):193--219, 2013.

\bibitem[{\v{S}}{\v{t}}o14]{St13-ICRA}
Jan {\v{S}}{\v{t}}ov{\'{\i}}{\v{c}}ek.
\newblock Exact model categories, approximation theory, and cohomology of
  quasi-coherent sheaves.
\newblock In David~J. Benson, Henning Krause, and Andrzej Skowro{\'n}ski,
  editors, {\em Advances in Representation Theory of Algebras ({C}onf. {ICRA}
  {B}ielefeld, {G}ermany, 8-17 {A}ugust, 2012)}, EMS Series of Congress
  Reports, pages 297--367. EMS Publishing House, Z{\"u}rich, 2014.

\bibitem[YL11]{YL11}
Gang Yang and Zhongkui Liu.
\newblock Cotorsion pairs and model structures on {${\rm Ch}(R)$}.
\newblock {\em Proc. Edinb. Math. Soc. (2)}, 54(3):783--797, 2011.

\end{thebibliography}
